\documentclass[11pt]{amsart}
\usepackage{lmodern}
\usepackage{amsmath, amsthm, amssymb, amsfonts}
\usepackage[normalem]{ulem}
\usepackage{mathrsfs}
\usepackage{verbatim} 
\usepackage{longtable}
\usepackage{mathtools}

\usepackage{tikz}
\usetikzlibrary{decorations.pathmorphing}
\tikzset{snake it/.style={decorate, decoration=snake}}
\usepackage{caption}
\usepackage{tikz-cd}
\usetikzlibrary{arrows}
\usepackage{enumerate}
\usepackage{stmaryrd} 
\usepackage{graphicx}%
\usepackage[hidelinks]{hyperref}
\usepackage{multirow}%
\usepackage[title]{appendix}%
\usepackage{xcolor}%
\usepackage{textcomp}%
\usepackage{manyfoot}%
\usepackage{booktabs}%
\usepackage{listings}%
\usepackage[utf8]{inputenc}
\hypersetup{colorlinks,linkcolor=blue}

\theoremstyle{plain}
\newtheorem{thm}{Theorem}[section]
\newtheorem{cor}[thm]{Corollary}

\newtheorem{prop}[thm]{Proposition}
\newtheorem{conj}[thm]{Conjecture}

\theoremstyle{definition}
\newtheorem{defn}[thm]{Definition}
\newtheorem{example}[thm]{Example}

\theoremstyle{remark}
\newtheorem{rmk}[thm]{Remark}

\newcommand{\BC}{{\mathbb{C}}}

\newcommand{\BN}{{\mathbb{N}}}

\newcommand{\BP}{{\mathbb{P}}}
\newcommand{\BQ}{{\mathbb{Q}}}

\newcommand{\BU}{{\mathbb{U}}}

\newcommand{\CC}{{\mathcal C}}

\newcommand{\CJ}{{\mathcal J}}

\newcommand{\CO}{{\mathcal O}}

\newcommand{\Fp}{{\mathfrak{p}}}
\newcommand{\Fq}{{\mathfrak{q}}}
\newcommand{\Fr}{{\mathfrak{r}}}

\DeclareFontFamily{OT1}{rsfs}{}
\DeclareFontShape{OT1}{rsfs}{n}{it}{<-> rsfs10}{}
\DeclareMathAlphabet{\curly}{OT1}{rsfs}{n}{it}

\newcommand{\Chow}{\mathrm{CH}}
\newcommand{\Corr}{\mathrm{Corr}}

\newcommand{\cjac}{\overline{J}_{C}}
\newcommand{\cjacg}{\overline{\mathcal{J}}_{C}}

	\newcommand{\mdol}{\mathcal{M}_{Dol}}

	\newcommand{\mb}{\mathcal{M}_{B}}

 \newcommand{\Aff}{\mathbb{A}}

\usepackage{tikz}
\usepackage{lmodern}
\usetikzlibrary{decorations.pathmorphing}

\newcommand{\ita}{\textit}
\newcommand{\gr}{\textbf}
\makeatletter
\let\@wraptoccontribs\wraptoccontribs
\title{P=W phenomena in algebraic and enumerative geometry}
\author{Camilla Felisetti}
	
\begin{document}
		\maketitle
    \begin{abstract}
        In view of the recent proofs of the P=W conjecture, the present paper reviews and relates the latest results in the field, with a view on how P=W phenomena appear in multiple areas of algebraic geometry. As an application, we give a detailed sketch of the proof of P=W by Maulik, Shen and Yin.
    \end{abstract}
		{\footnotesize{\textbf{Keywords:} Higgs bundles, character varieties, $P=W$ conjecture, non abelian Hodge theory, hyperk\"ahler manifolds}}\\
		\noindent
{\footnotesize{\textbf{MSC 2020 classification:} 14D20, 14D06, 32S60 }}
	\tableofcontents
		\linespread{1.2}
  
\section{Introduction to P=W phenomena}

P=W phenomena provide a unified interpretation of the symmetries enjoyed by the cohomology groups of certain complex algebraic varieties. Let $X$ be a smooth complex algebraic variety of dimension $n$.
If $X$ is projective and $\eta\in H^2(X,\mathbb{Q})$ is an ample class then, for every $i$, the Hard Lefschetz theorem establishes isomorphisms
\begin{equation}\label{eq:HL}
\cup \eta^i:\ H^{2n-i}(X,\mathbb{Q})\xrightarrow{\simeq}H^{2n+i}(X,\mathbb{Q}).
\end{equation}
Given any smooth algebraic map $f:X\rightarrow Y$ of relative dimension $n$, we can interpret the cohomology of $X$ as the total hypercohomology of the complex $Rf_*\mathbb{Q}_X$ on $Y$
\begin{equation*}
H^*(X,\mathbb{Q})=H^*(Y,Rf_*\mathbb{Q}),
\end{equation*}
and the Hard Lefschetz theorem is a direct consequence of its relative version 
$$  \cup \eta^i: \ R^{n-i}f_*\mathbb{Q}_X\xrightarrow{\simeq}R^{n+i}f_*\mathbb{Q}_X,$$
so that \eqref{eq:HL} can be refined into isomorphisms
\begin{equation}\label{eq:RHLsmooth}
\cup \eta^i: \ \mathrm{Gr}_L^{n-i}H^d(X,\mathbb{Q})\xrightarrow{\simeq} \mathrm{Gr}_L^{n+i}H^{d+2i}(X,\mathbb{Q}),
\end{equation}
where $L_{\bullet}$ denotes the Leray filtration associated with $f$.

When the map $f:X\rightarrow Y$ is no longer a smooth morphism, \eqref{eq:RHLsmooth} continues to hold up to replacing the Leray filtration with the \textit{Perverse Leray filtration} $P_{\bullet}$ associated with $f$:
$$ \cup \eta^i: \ \mathrm{Gr}_P^{n-i}H^d(X,\mathbb{Q})\xrightarrow{\simeq} \mathrm{Gr}_P^{n+i}H^{d+2i}(X,\mathbb{Q}).$$
The perverse filtration $P_\bullet$ on $H^*(X,\mathbb{Q})$ is defined in the same way as the classical Leray filtration, but replacing the standard truncation functors $\tau_{\leq k}$ with the $\ita{perverse}$ truncation functors $^\mathfrak{p}\tau_{\leq k}$ (see the appendix for further details): 
$$P_kH^{n+j}(X,\mathbb{Q}):=\mathrm{Im}\left\lbrace H^j(Y,\ ^\mathfrak{p}\tau_{\leq k}Rf_*\mathbb{Q}_X[n]) \rightarrow H^j(Y,Rf_*\mathbb{Q}_X[n])\right\rbrace, \quad \forall j, k\in \mathbb{Z}.$$
If the target $Y$ is either affine or projective (as it happens for the cases of interest in this exposition),  $P_{\bullet}$ can be described with a particularly simple \textquotedblleft flag\textquotedblright \ form:
\begin{prop}[\cite{deCataldoMigliorini2010}, Theorem 4.1.1 and \cite{deCataldoMigliorini05}, Proposition 5.2.4 ] Let $f:X\rightarrow Y$ be a proper map of algebraic varieties and let $P$ be the associated perverse filtration. 
\begin{enumerate}
\item If $Y$ is affine then $$ P_k H^d(X,\mathbb{Q})= \mathrm{Ker}\left\lbrace H^d(X,\mathbb{Q})\rightarrow H^d(f^{-1}(\Lambda^{d-k-1}),\mathbb{Q})\right\rbrace,$$
where $\Lambda^{i}$ denotes any $i$-dimensional linear section of $Y\subseteq \Aff^N$.
\item If $Y$ is projective and $L=f^*\alpha$ is the pullback of an ample class on $Y$, then 
$$P_kH^d(X, \BQ) = \sum_{i\geq 1} \left(
\mathrm{Ker}(L^{n+k+i-d}) \cap \mathrm{Im}(L^{i-1}) \right) \cap H^d(X, \BQ).$$
\end{enumerate}
\end{prop}

In general, in absence of some degree of properness by $X$ or by the map $f$, we do not expect any simmetries of Hard Lefschetz type on the cohomology of $X$. Surprisingly, this kind of symmetries has been first observed on the cohomology of a non compact symplectic variety, namely the character variety of (twisted) representations of the fundamental group of a curve into some general linear group. This has led to the formulation of the P=W conjecture and, more generally, to the investigation of P=W phenomena.

\subsection{The original P=W conjecture}\label{sec:theoriginalconjecture}

Let $\Sigma$ be a smooth projective curve of genus $g\geq 1$ and fix a point $p\in \Sigma$. Fix two integers $r,n\in \mathbb{Z}$ with $r>0$ and $(r,n)=1$.  We define the \ita{Betti moduli space}
\begin{equation*}\label{eq:defbetti}
\mb(r,n):=\left\lbrace  A_1,\ldots,A_g,B_1\ldots B_g\in\mathrm{GL}_r(\mathbb{C})^{2g}\mid \prod_{i=1}^g[A_i,B_i]=e^{\frac{2\pi i n}{r} }I_r\right\rbrace\sslash \mathrm{GL}_r(\mathbb{C})
\end{equation*}
as the (twisted) character variety of representations of the fundamental group $\pi_1(\Sigma,p)$ of $\Sigma$ into $\mathrm{GL}_r(\mathbb{C})$.

The Betti moduli space is an affine GIT quotient and its cohomology groups carry a mixed Hodge structure, which turns out to be of Hodge--Tate type, with weight filtration $W_\bullet$. 
As mentioned before, although $\mb(r,n)$ is not proper, the graded pieces of the weight filtration on $H^*(\mb(r,n),\mathbb{Q})$ satisfy unexpected Hard Lefschetz symmetry, called the \ita{curious Hard Lefschetz property}.
\begin{thm}[Curious Hard Lefschetz] There is a class $\sigma\in H^2(\mb(r,n),\mathbb{Q})$ of weight $4$ such that 
\begin{equation*}
\cup \sigma^i: \mathrm{Gr}_W^{\dim \mb-2i}H^d(\mb(r,n),\mathbb{Q})\xrightarrow{\simeq} \mathrm{Gr}_W^{\dim \mb+2i}H^{d+2i}(\mb(r,n),\mathbb{Q}).
\end{equation*}
\end{thm}
The curious Hard Lefschetz theorem was first observed in rank 2 by Hausel and Rodriguez-Villegas \cite{HauselRodriguez-Villegas2008} and  proved by Mellit in full generality in \cite{Mellit2019}.
De Cataldo, Hausel and Migliorini \cite{deCataldoHauselMigliorini2012} conjectured that the curious Hard Lefschetz symmetry could be explained in terms of a real analytic isomorphism of $\mb(r,n)$ with another moduli spaces of very different nature, the \ita{Dolbeault moduli space}
$$\mdol(r,n)=\left\lbrace (E,\theta) \text{ semistable Higgs bundle of rank }r \text{ and degree }n\right\rbrace/\sim_S,$$ parametrizing equivalence classes of semistable Higgs bundles of rank $r$ and degree $n$.
It is well known that, under the coprimality hypothesis, $\mdol(r,n)$ is a nonsingular quasi-projective variety of dimension $2(r^2(g-1)+1)$ endowed with a holomorphic symplectic form and it is equipped with a proper Lagrangian map $h: \mdol(r,n)\rightarrow \mathcal{A}\cong\mathbb{A}^{r^2(g-1)+1}$, which is called \ita{Hitchin fibration}.

In fact, any Higgs field $\theta$ is a twisted endomorphism of a vector bundle $E$, so it has a well defined characteristic polynomial, whose coefficients are  $\mathrm{tr}(\theta)\in H^0(C,\Omega^1_{\Sigma})$, $\mathrm{tr}(\Lambda^2\theta)\in H^0(\Sigma,(\Omega^1_{\Sigma})^{\otimes 2}), \ldots ,\det(\theta)\in H^0(\Sigma,(\Omega^1_{\Sigma})^{\otimes n})$.
This defines a map
$$	h:  \mdol(r,n) \to  \mathcal{A}, \quad (E,\theta) \mapsto \mathrm{charpol}(\theta).$$
The celebrated Beauville-Narasimhan-Ramanan (BNR) \cite{BNR89} correspondence states that the fiber $h^{-1}(a)$ of $h$ over a general point $a\in \mathcal{A}$ is isomorphic to the compactified Jacobian $\overline{J}_{C_a}$ of a branched covering $C_a$ of $\Sigma$, called \ita{spectral curve}.

The following theorem is the coronating result of non abelian Hodge theory and is understood as a sum of results of Donaldson, Corlette, Hitchin and Simpson, see \cite{Donaldson83,Corlette88, Hitchin87, Simpson90}.
\begin{thm}[Non abelian Hodge theorem]
For all $r$ and $n$ as above there exists a real analytic isomorphism 
$$\mdol(r,n)\cong \mb(r,n).$$
\end{thm}

To explain the correspondence between curious Hard Lefschetz and the classical relative Hard Lefschetz theorem for $h$, De Cataldo, Hausel and Migliorini \cite{deCataldoHauselMigliorini2012} conjectured that the non abelian Hodge theorem should exchange the weight filtration on $H^*(\mb(r,n),\mathbb{Q})$ with the perverse filtration associated with the Hitchin map $h$ on $H^*(\mdol(r,n),\mathbb{Q})$ up to a trivial renumbering. 		
\begin{conj}[P=W conjecture] Let $r$ and $n$ be coprime integers and
	$$\psi: \mb(r,n)\rightarrow \mdol(r,n)$$
	be the real analytic isomorphism of the non abelian Hodge theorem. Then the associated isomorphism in cohomology 
	$$ \psi^*:H^*(\mdol(r,n),\mathbb{Q})\rightarrow H^*(\mb(r,n),\mathbb{Q})$$
	is such that, for all $k\in \mathbb{Z}$, 
	$$P_{k}H^*(\mdol(r,n),\mathbb{Q})\xrightarrow{\simeq} W_{2k}H^*(\mb(r,n),\mathbb{Q}).$$
\end{conj}

\par\noindent 
\gr{Notation.} For sake of readability, we fix the coprime integers $r,n$ and denote $\mdol(r,n)$ and $\mb(r,n)$ simply by $\mdol$ and $\mb$.  We will specify $r$ and $n$ in the notation only when they are relevant to the context.  

\subsection{Tautological classes}
		
In 2019, de Cataldo, Maulik and Shen  \cite{deCataldoMaulikShen2019} reduced the proof of P=W to the multiplicativity of the perverse filtration. One key tool is the generation of the cohomology of $\mdol$ by tautological classes due to Markman \cite{Markman2002}.
In particular, Markman described the generators of the cohomology of $\mdol$ in terms of the Chern character of the (normalized) universal bundle $\mathcal{U}$ on $\mdol\times \Sigma$ .
In particular given the diagram 
		\begin{center}
			\begin{tikzcd}
		& \mdol\times \Sigma \arrow[ld, "p_{\mathcal{M}}"'] \arrow[rd, "p_C"] &   \\
		\mdol &                                        & C
	\end{tikzcd}
	\end{center}

one defines for each $\gamma\in H^i(\Sigma,\mathbb{Q})$  a \ita{tautological class}
\begin{equation}\label{eq:taugen}
c_k(\gamma)=p_{\mathcal{M},*}(\mathrm{ch}_k(\mathcal{U})\cup p_\Sigma^{*}\gamma)\in H^{i+2k-2}(\mdol,\mathbb{Q})
\end{equation}
and proves that  $$\left\lbrace c_k(\gamma) \mid \gamma \in H^*(\Sigma,\mathbb{Q}), \ k \geq 0 \right\rbrace$$
is a set of generators for $H^{*}(\mdol,\mathbb{Q})$ and, consequently, so is for $H^*(\mb,\BQ)$ through the non abelian Hodge isomorphism.
		
On the Betti side, the weight filtration on $\mb$ has been computed by Shende in \cite{Shende17}. In particular he proves the following. 
\begin{thm}[Weight of tautological classes] The mixed Hodge structure on $H^*(\mb,\mathbb{Q})$ is of Hodge-Tate type and for all $\gamma \in H^i(\Sigma,\mathbb{Q})$ one has that 
$$ c_k(\gamma)\in \ ^k\mathrm{\mathrm{Hdg}}^{i+2k-2}(\mb).$$
Here $^k\mathrm{\mathrm{Hdg}}^{d}(\mb):= W_{2k}H^d(\mb,\mathbb{Q})\cap F_k H^d(\mb,\mathbb{C})\cap \overline{F}^d(\mb,\mathbb{C})$ denotes the classes of type $(k,k)$ in $H^d(\mb,\mathbb{Q})$.
\end{thm}
As a result there is a canonical decomposition of graded vector spaces
$$ H^*(\mdol,\mathbb{Q})\cong H^*(\mb,\mathbb{Q})=\bigoplus_{k,i} \ ^{k}\mathrm{Hdg}^i(\mb)$$
and the P=W conjecture can be rephrased as 
$$ P_k H^*(\mdol,\mathbb{Q})=\bigoplus_{k'\leq k} \ ^{k'}\mathrm{Hdg}^*(\mb).$$
Hence, one can split the resolution of the P=W conjecture in two separate problems (see \cite[Conjecture 0.3]{deCataldoMaulikShen2019}) on the Dolbeault moduli space.
\begin{prop}[Equivalent version of P=W]\label{conj:equivpw} The P=W conjecture is equivalent to the following statements 
\begin{enumerate}[(i)]
\item (Tautological classes) The tautological classes $c_k(\gamma)\in H^*(\mdol,\mathbb{Q})$ have perversity $k$ for all $k\geq 0$ and all $\gamma \in H^*(\Sigma,\mathbb{Q})$;
\item (Multiplicativity)  The perverse filtration is multiplicative, i.e. for all $l,k$ 
$$ \cup: P_kH^*(\mdol,\mathbb{Q})\times P_lH^*(\mdol,\mathbb{Q})\rightarrow P_{p+q}H^*(\mdol,\mathbb{Q}).$$
\end{enumerate} 
\end{prop}
Since in the same paper de Cataldo, Maulik and Shen prove the first item of Proposition \ref{conj:equivpw}, thus showing that the proof of the P=W conjecture reduces to that of the multiplicativity of the perverse filtration in the second item. 

\begin{rmk}
While the weight filtration is naturally multiplicative, in general the perverse filtration is not, see \cite[Exercise 5.6.8]{deCataldo2016}. Indeed, etablishing multiplicativity of $P_\bullet$ is one the most challenging points in all existing proofs of P=W. 
\end{rmk}

\subsubsection{History of proofs of P=W}  
The P=W conjecture has been proved in 2010 in rank 2 in the original paper by de Cataldo, Hausel and Migliorini \cite{deCataldoHauselMigliorini2012} and by the same authors in genus 1 and arbitrary rank \cite{deCataldoHauselMigliorini2013}. 
The conjecture remained untackled for almost ten years when de Cataldo, Maulik and Shen proved it for genus 2 and arbitrary rank \cite{deCataldoMaulikShen2019}. Moreover, an enumerative approach was proposed by Chiarello, Hausel and Szenes in \cite{CHS2020}.
Later, two different proofs of the conjecture in full generality appeared in September 2022: the first is due to Maulik and Shen \cite{MaulikShen2022}, while the second is due to Hausel, Mellit, Minets and Schiffmann \cite{HMMS2022}. 
The last proof appeared a year later and is due to Maulik, Shen and Yin \cite{MSY}.
While here we do not treat the proofs of Maulik-Shen and Hausel-Mellit-Minets-Schiffmann as they are reviewed in the excellent Séminaire Bourbaki paper \cite{surveyciky} (see also \cite{Mauri23, MiglioUmi}), in section \ref{sec:pc} we give a detailed sketch of the proof by Maulik, Shen and Yin \cite{MSY} as it opens up to a generalization of P=W in the more general context of Perverse=Chern phenomena. 

\subsubsection{Generalizations of the conjecture to the non coprime case}\label{sec:piwi}
		
	 One may wonder whether the picture described in the previous section holds for the honest character variety of $\mathrm{GL}_r(\mathbb{C})$ which corresponds to the Higgs bundles of rank $r$ and degree 0. 
		In general, when $r$ and $n$ are not coprime the moduli spaces are no longer smooth and the cohomology groups of $\mdol(r,n)$ and $\mb(r,n)$ do not satisfy the relative and curious Hard Lefschetz theorems anymore.
		
		Nonetheless, it is known that the relative Hard Lefschetz theorem for $h$ holds for \ita{intersection cohomology} $IH^*(\mdol(r,n),\mathbb{Q})$, see \autoref{sec:appendix}. \\
		Intersection cohomology is a generalization of cohomology designed to restore the Hodge theoretic properties which cease to hold for the cohomology of singular varieties, such as Poincaré duality. Moreover, the intersection cohomology groups of an algebraic variety carry a mixed Hodge structure and, if the variety is projective, they satisfy Hard Lefschetz theorem and admit a pure Hodge structure.
		
		This suggests that a natural formulation for the P=W conjecture in the singular case would involve this invariant.
		Another evidence of this fact was provided by de Cataldo and Maulik in  \cite{deCataldoMaulik2018}, where they proved that the perverse filtration on intersection cohomology is independent of the complex structure of the curve $\Sigma$, exactly as it happens for the weight filtration. As a result, they conjectured \cite[Question 4.1.7]{deCataldoMaulik2018} the following statement.
		
		\begin{conj}[PI=WI conjecture]\label{PI=WIconj} Let $r,n$ be two non necessarily coprime integers with $r>0$ Let $\psi: \mb(r,n)\rightarrow \mdol(r,n)$
			be the real analytic isomorphism of the non abelian Hodge theorem. Then the associated isomorphism in intersection cohomology 
			$$ \psi^*:IH^*(\mdol(r,n),\mathbb{Q})\rightarrow IH^*(\mb(r,n),\mathbb{Q})$$
			is such that, for all $k\in \mathbb{Z}$, 
			\[P_{k} IH^*(\mdol,(r,n) ,\mathbb{Q})\xrightarrow{\simeq} W_{2k}IH^*(\mb(r,n), \mathbb{Q}).\]
		\end{conj}
		
		At present, the PI=WI conjecture has been proved for $d=0$ by the author and Mauri in \cite{FelisettiMauri} for moduli spaces which admits a symplectic resolution of singularities, namely when $g=1$ and $r$ is arbitrary and $g=2$, $r=2$. Remarkably, these cases can be viewed as degenerations of the known few examples of \ita{irreducible symplectic manifolds}\footnote{Roughly speaking, a manifold is irreducible symplectic or, equivalently, hyperk\"ahler if it has three complex structures interacting quaternionically, see \cite{Beauville84} for a precise definition.} up to deformation, as we will see in more detail in \autoref{sec:pwcompact}.
		
		\begin{rmk}
		Recently Davison showed that, under some conjectural hypotheses, the P=W conjecture is equivalent to the PI=WI conjecture by a phenomenon called \ita{$\chi$-independence}, which will be treated in \autoref{sec:pwenumerative}.
		\end{rmk}

\subsection{P=W phenomena}

After the P=W conjecture was formulated, the attempts to prove it have led to consider the P=W statement under different perspectives giving rise to P=W phenomena. 
Very vaguely speaking, a P=W phenomenon should involve:
\begin{itemize}
    \item A complex variety (or a stack) $M$ of even dimension $2N$ and a proper algebraic map $f:M\to B$: this is required so that some cohomology theory of $M$ (e.g. Betti cohomology, intersection cohomology, Borel--Moore homology) satisfies the relative Hard Lefschetz theorem;
    \item An isomorphism $\Psi$ between the cohomology $\mathbb{H}$ of $M$ and the cohomology $\mathbb{H}'$ of another variety (or even an endomorphism of the cohomology of $M$) such that $\mathbb{H}'$ has mixed Hodge structure (or simply a filtration $W_\bullet$) which satisfies the curious Hard Lefschetz property. 
    \item A correspondence $P_k\mathbb{H}\cong W_{2k}\mathbb{H}'$ induced by $\Psi$.
\end{itemize}

 We here choose to focus on three different manifestations of P=W phenomena:
\begin{itemize}
\item P=W in the compact symplectic case;
\item stacky P=W conjecture and enumerative invariants;
\item Perverse=Chern phenomena.
\end{itemize}

Other P=W phenomena involve, for example, the \ita{geometric} P=W conjecture \cite{KatzarkovNollPanditEtAl2015, MMS,Szabo2018}; a P=W for Painlevé space \cite{Szabo2019, NemethiSzabo2020}; an analogue of the P=W conjecture for compact abelian varieties \cite{BKU2023}.

\subsection*{Acknowledgments}
The author wishes to thank Lucien Hennecart for answering many questions regarding enumerative P=W phenomena and the anonimous referee for the precious comments to the first draft of this paper.
\section{P=W phenomena in the compact setting}\label{sec:pwcompact}
One of the key ingredients in reducing proof of P=W to establishing the multiplicativity of $P_\bullet$ with respect to cup product is the fact that one can view the Dolbeault moduli space as a special fiber of a \textquotedblleft degeneration\textquotedblright\ having a compact irreducible symplectic manifold as generic object. Here by degeneration we mean a flat (not necessarily proper) morphism of normal algebraic varieties, typically over a curve. 
		
		\subsection{Degeneration of Hitchin systems to Mukai systems}
		The proposal of using the extensively studied geometry of irreducible symplectic manifolds to understand that of moduli spaces of non abelian Hodge theory finds its roots in a work by Donagi, Ein and Lazarsfeld in 1997
		\cite{DEL97}, where they exhibit the first instance of the aforementioned degenerations. For a more detailed and general description see for instance \cite{deCataldoMaulikShen2019,deCataldoMaulikShen2020,FelisettiMauri}.
		
		The compact irreducible symplectic manifolds appearing in the degenerations are Mukai moduli spaces of sheaves on a K3 (or abelian) surface. Given any coherent sheaf $\mathcal{F}$ on a K3 surface $S$, one can associate to it a \ita{Mukai vector} $$v = (\mathrm{rk}(\mathcal{F}), c_1(\mathcal{F}), \chi(\mathcal{F})-\mathrm{rk}(\mathcal{F})) \in H^{*}_{\mathrm{alg}}(S, \mathbb{Z}).$$ We denote by $M_v(S)$ the moduli space of Gieseker semistable sheaves on $S$ with Mukai vector $v$ for a sufficiently general polarization $H$ (which is tipically omitted in the notation); see \cite[\S 1]{Simpson1994I}.  
		
		The curve $\Sigma$ embeds in $S$ as an ample divisor with the map
		\begin{equation*}
			j: \Sigma \hookrightarrow S.
		\end{equation*}
		As in \cite{DEL97}\footnote{In \cite{DEL97}, the divisor is supposed very ample, but this assumption can in fact be dropped.}, we consider the degeneration of $S$ to the normal cone of $\Sigma$ in $S$ (see, e.g. \cite{Fulton98}) which is given by the family
		$$ \mathcal{S}=\left(\mathrm{Bl}_{\Sigma\times 0}S\times \Aff^1\right)\setminus \left( S\times 0\right)\rightarrow \Aff^1.$$
		The central fiber $\mathcal{S}_0$ is isomorphic to the cotangent space $T^*\Sigma$: in fact, since $K_S$ is trivial, the adjunction formula implies that the normal bundle of $\Sigma$ in $S$ coincides with $K_{\Sigma}$, whose total space is $T^*\Sigma$. 
		The restriction to $\Aff^1 \setminus \{0\}$ is a trivial fibration
		$S\times (\Aff^1 \setminus \{0\})\rightarrow \Aff^1 \setminus \{0\}$.
		Take a relative compactification $\mathcal{S} \subset \overline{\mathcal{S}}$ over $\Aff^1$. For all $t\in \Aff^1$, set $\beta_t=r[\Sigma]\in H_2(\mathcal{S}_t,\mathbb{Z})$ with $r>0$ and consider  $$\mathcal{M} \to \Aff^1,$$ the coarse relative moduli space of one-dimensional Gieseker semistable sheaves $\mathcal{F}$ whose support is proper and contained in  $\mathcal{S}_t \subseteq \overline{\mathcal{S}}_t$ with $\chi(\mathcal{F})=\chi$ and $[\mathrm{Supp}\mathcal{F}]=\beta_t$; see \cite[Theorem 1.21]{Simpson1994I}. 
		Choosing $\chi$ appropriately, the central fiber recovers the Dolbeault moduli space  
		$$\mathcal{M}_0 \simeq \mdol(r,n).$$
		In fact the moduli space of Higgs bundles on $\Sigma$ of rank $r$ and degree $n$ can be realized via the BNR-correspondence \cite{BNR89} as the moduli space of one-dimensional Gieseker-semistable sheaves $\mathcal{F}$ on $T^*\Sigma$ with $\chi(\mathcal{F})=\chi$ and $[\mathrm{Supp}\mathcal{F}]=\beta_0$. The general fiber is isomorphic to 
		\[\mathcal{M}_t \simeq M_v(S)\]
		with Mukai vector $v=(0,r\Sigma, \chi)$.
		In this degeneration, if $\dim M_v(S)=2N$, the natural forgetful Lagrangian fibration
		$$\pi: M_v(S)\rightarrow \mathbb{P}^{N}, \quad \mathcal{F}\mapsto [\mathrm{Supp}\mathcal{F}]$$
		is sent to the Hitchin fibration $h:\mdol(r,n)\rightarrow \mathcal{A}$.\\
		
		It is then natural to ask the following question: 
		
		\gr{Question: } Can we see a manifestation of the P=W correspondence in the compact setting? More generally, given a compact irreducible symplectic manifold with a Lagrangian fibration, do we have a P=W phenomenon?\\
		
		\subsection{P=W phenomena for irreducible holomorphic symplectic manifolds}
		The answer to this question was first provided by Shen and Yin in \cite{ShenYin2018}. 
		Given a nonsingular irreducible symplectic manifold $M$ with a Lagrangian fibration $\pi:M\rightarrow B$, the cohomology groups of $M$ admit a pure Hodge structure.  Moreover, one can consider the perverse filtration on $H^*(M,\mathbb{Q})$ associated with $\pi$. As a result, we have well defined  Hodge numbers and perverse Hodge numbers as 
        $$ h^{i,j}(M)=\dim Gr_{F}^i H^{i+j}(M,\mathbb{C}),\quad \ ^{\mathfrak{p}}h^{i,j}(M)=\dim Gr_P^i H^{i+j}(M,\mathbb{C}).$$		
		Using a deformation argument on the \ita{period domain} $D$ parametrizing all possible hyperk\"ahler structures on $M$, Shen and Yin prove that, given a compact irreducible symplectic manifold with a Lagrangian fibration, the perverse numbers associated with the fibration match with the Hodge numbers of the total space, establishing a Perverse=Hodge phenomenon.
		\begin{thm}\label{thm:numpwcompact} Let $\pi:M\rightarrow B$ be a Lagrangian fibration from a nonsingular compact irreducible symplectic manifold to an algebraic variety $B$ and let $P_\bullet$ be the perverse filtration associated with $\pi$. Then
			\begin{equation} \ ^\mathfrak{p} h^{i,j}(M)=h^{i,j}(M).
			\end{equation}
		\end{thm}
		
		The above result has applications in several areas of algebraic geometry: for example it can be used to determine that the cohomology of the base of the Lagrangian fibration $\pi$ is that of a projective space, providing an answer to a question of Matsushita \cite{Matsushita99}; also the equality between perverse and Hodge numbers allows to compute enumerative invariants associated with specific Calabi-Yau threefolds, see \autoref{sec:pwenumerative}.
		
		Later, Harder, Li, Shen and Yin refined the result of Theorem \ref{thm:numpwcompact} by identifying the perverse filtration of a Lagrangian fibration on $M$ with the monodromy weight filtration of a maximally unipotent degeneration of compact irreducible symplectic manifolds \cite{HLSY2019}.
		
		In particular, the result of \cite{HLSY2019} implies the multiplicativity of the perverse filtration with respect to cup product on $H^*(M,\mathbb{Q})$. 
		
		\begin{thm}[Multiplicativity of the perverse filtration in the compact setting]
			Let $\pi:M\rightarrow B$ be a Lagrangian fibration from a nonsingular compact irreducible symplectic manifold to an algebraic variety $B$. The the perverse filtration $P$ associated with $\pi$ is multiplicative under cup product, i.e.
			$$\cup: \  P_kH^dM,\mathbb{Q})\times P_lH^{d'}(M,\mathbb{Q})\rightarrow P_{k+l}H^{i+j}(M,\mathbb{Q}).$$
		\end{thm}

\begin{rmk} A categorification of Theorem \ref{thm:numpwcompact} formulated in terms of quasi-isomorphisms of complexes on the base $B$ has been conjectured in \cite{ShenYin23} and proved in \cite{Schnell23}.
\end{rmk}
\begin{rmk}
  Theorem \ref{thm:numpwcompact} has been generalized by Shen, Yin and the author to singular irreducible symplectic varieties which admit a symplectic resolution, again replacing cohomology by intersection cohomology and providing a compact counterpart of the PI=WI conjecture \cite{FelisettiShenYin2021}. Moreover, there are partial results when $M$ has isolated singularities \cite{Tighe22}.
\end{rmk}
  \begin{rmk}
  The P=W phenomena in the compact case, the singular generalizations of P=W and the proof of the Curious Hard Lefschetz theorem \cite{Mellit2019} suggest that the existence of a symplectic structure on $M$ should be a key ingredient in finding P=W phenomena. However, somehow remarkably, the symplectic structure on the moduli spaces plays no essential role in all existing proofs of the P=W conjecture. This may have an explanation in view of the P=C phenomena treated in section \ref{sec:pc}.
  \end{rmk}

\section{P=W phenomena in enumerative geometry}\label{sec:pwenumerative}
		
\subsection{From P=W phenomena to  enumerative geometry counting invariants}
		P=W phenomena have tremendous impact in enumerative geometry: for instance, perverse Hodge numbers play an important role in recent constructions of curve counting invariants.
		
		In the late 90's,  Gopakumar and Vafa defined numerical invariants $n_{g,\beta}\in\mathbb{Z}$  for a Calabi–Yau 3-fold $Y$ and a curve class $\beta\in H^2(Y, \mathbb{Z})$,  called \ita{GV invariants} or \ita{BPS invariants}.
		These invariants count genus $g$ curves on $Y$ lying in the curve class $\beta$, see \cite{GV}. Moreover they are expected to give an equivalent counting to that of the Gromov-Witten invariants (see \cite{PandariphandeThomas2014} for a detailed account of this fact).
		
	    In 2018, Maulik and Toda \cite{MaulikToda2018} proposed a definition of BPS-invariants via the Hilbert-Chow morphism 
		$$\mathrm{hc}:\mathcal{M}_{\beta}\rightarrow \mathrm{Chow}_{\beta}(Y)$$
		from the moduli space of 1-dimensional stable sheaves on $Y$ with support $\beta$ to the corresponding Chow variety. When $\mathcal{M}_{\beta}$ is nonsingular, Maulik and Toda prove an equality of power series
		\begin{equation} \label{GV}
			\sum_{i\in \mathbb{Z}}\chi( \ ^{\mathfrak{p}}\mathcal{H}^i(R\mathrm{hc}_*\mathbb{Q}_{\mathcal{M}_{\beta}}[\dim \mathcal{M}_{\beta}]))y^i=\sum_{g\geq 0}n_{g,\beta}\left(y^{\frac{1}{2}}+y^{-\frac{1}{2}}\right)^{2g}.
			\end{equation}
		The quantity on left hand side of \eqref{GV} is the Euler characteristic\footnote{For any bounded  constructible complex $K$ on an algebraic variety $Y$, one sets $\chi(K):= \sum (-1)^j \dim H^j(Y,K)$} of the \ita{perverse cohomology sheaves} (see \autoref{sec:appendix}) arising in the decomposition theorem of the map $\mathrm{hc}$ and it can be computed in terms of the perverse hodge numbers $^{\mathfrak{p}}h^{i,j}(\mathcal{M_{\beta}})$. \\
		If one considers $Y=S\times \mathbb{C}$ for a K3 surface $S$, then the topology of the map $\mathrm{hc}$ is uniquely determined by that of $\pi:M_v(S)\rightarrow \mathbb{P}^N$ where the Mukai vector on $S$ is chosen appropriately (so that $\beta$ is the support of the sheaves and $M_v(S)$ is nonsingular).  In that case $M_v(S)$ is deformation equivalent to the $N$-th Hilbert scheme of S, for 
	$N=\frac{1}{2}\beta^2+1$, so its Hodge numbers (and thus the perverse!) are known. 
		
		In other words, Theorem \ref{thm:numpwcompact} provides a direct calculation of the BPS invariants of a K3 surface in terms of the Hodge numbers of the Hilbert schemes of points on a K3 surface, see \cite{ShenYin2018} and \cite{deCataldoMaulikShen2019} for a detailed discussion on this fact.
		
		\subsection{From enumerative geometry to P=W phenomena}
		
		In the other direction, it is not to be underestimated the large amount of evidences that enumerative geometry has offered to the P=W conjecture, its variants and more generally to P=W phenomena.
		
		Let us start by an easy observation: two Betti moduli spaces $\mb(r,n)$ and $\mb(r,n')$ are Galois conjugate when $\gcd(r,n)=\gcd(r,n')=1$, thus the algebraic isomorphism between them induces an isomorphism of mixed Hodge structures
		$$H^*(\mb(r,n),\BQ)\cong H^*(\mb(r,n'),\BQ)$$
		between their cohomology groups.
		Hence, the P=W conjecture suggests that the perverse filtration on $H^*(\mdol(r,n),\BQ)$ should be independent of $n$ as long as it is coprime with $r$.  
		
		Using methods coming from cohomological Donaldson-Thomas theory \cite{Szendroi2016}, in 2021  Kinjo and Koseki \cite{KinjoKoseki2021} (see also \cite{COW}) proved that this is in fact true, providing new evidences of the P=W conjecture. This phenomenon is usually referred to as \ita{$\chi$-independence} since an invariant of the moduli space of Higgs bundles, such as the perverse filtration, depends only on the rank $r$ of the Higgs sheaves and it is not affected by choice of their Euler characteristic $\chi$. 
		The following result is \cite[Theorem 1.1]{KinjoKoseki2021}.
		\begin{thm}[$\chi$-independence for $\mdol(r,n)$]\label{chiindepcoprime}
			Let $r,n,n'$ be integers such that $r>0$ and $\gcd(r,n)=\gcd(r,n')=1$. Then there exists an isomorphism 
			$$ H^*(\mdol(r,n),\BQ)\cong  H^*(\mdol(r,n'),\BQ)$$
			preserving Hodge structures and the perverse filtration. 
		\end{thm}
		
		When $r$ and $n$ are not coprime, as explained in section \ref{sec:piwi}, the natural invariant suggested by perverse sheaves theory is intersection cohomology. However,  motivated by physics, Chuang, Diaconescu and Pan \cite{CDP14} proposed to use \ita{BPS cohomology}: in fact, despite being just phisically defined, BPS cohomology is expected to enjoy a $\chi$-independence phenomenon without assuming coprimality between $r$ and $n$. Further, the BPS cohomology groups on both Dolbeault and Betti sides carry a Lie algebra structure (see \cite{Davison2020} and \cite{DHSM2022}) compatible with non abelian Hodge isomorphism \cite[Conjecture 1.5]{SS20}, opening the way to a possible representation theoretic approach to P=W problems.
		
		To overcome the difficulty of a rigorous mathematical definition, Kinjo and Koseki \cite{KinjoKoseki2021} propose a definition of BPS cohomology on $\mdol(r,n)$ using the cohomological Donaldson-Thomas theory of the 3-fold $\mathrm{Tot}(\mathcal{O}_\Sigma\oplus \Omega^1_\Sigma)$. Such as the intersersection cohomology of $\mdol(r,n)$ is defined as the hypercohomology of intersection complexes $\mathcal{IC}^{Dol}_{r,n}$, the BPS cohomology is the hypercohomology of another complex $\mathcal{BPS}^{Dol}_{r,n}$: 
		$$H^*_{BPS}(\mdol(r,n),\mathbb{Q}):=\mathbb{H}^{*}(\mathcal{BPS}^{Dol}_{r,n}).$$
		The complex $\mathcal{BPS}^{Dol}_{r,n}$ is a perverse sheaf and there is a natural injection
		$$ \mathcal{IC}^{Dol}_{r,n}\hookrightarrow \mathcal{BPS}^{Dol}_{r,n}$$
		which is an isomorphism in the coprime case.
		\begin{example} To see that in the non coprime case there is no isomorphism between the complexes it is sufficient to consider the moduli space $\mdol(2,0)$ on a genus 2 curve. Setting 
			\begin{align*}
				P^{IC}_t(X)&=\sum_i \dim(IH^i(X,\mathbb{Q}))t^i\\
				P^{BPS}_t(X)&=\sum_i \dim(H_{BPS}^i(X,\mathbb{Q}))t^i\,\
			\end{align*}
			one has that by \cite{Felisetti2018} and \cite{Rayan2018}
			\begin{align*}
				P^{IC}_t(\mdol(2,0))&=(1+t)^4(1+t^2+2t^4+2t^6)\\
				P^{BPS}_t(\mdol(2,0))&=(1+t)^4(1+t^2+4t^3+2t^4+4t^5+2t^6).\
			\end{align*}
			On the other hand, if one considers the coprime case of moduli space $\mdol(2,1)$ on a genus 2 curve then 
			$$ 	P_t(\mdol(2,1))=P^{IC}_t(\mdol(2,1))=P^{BPS}_t(\mdol(2,1))=P^{BPS}_t(\mdol(2,0))$$
			\end{example}
		
		Kinjo and Koseki generalize Theorem  \ref{chiindepcoprime} to BPS cohomology.
		\begin{thm}[\cite{KinjoKoseki2021}, Corollary 5.15]\label{chiindepnoncoprime}
			Let $r,n,n'$ be integers such that $r>0$. Denote by $h$ (resp. $h'$)$:\mdol(r,n)$(resp. $\mdol(r,n')$)$\rightarrow \mathcal{A}$ the Hitchin map.
			There exists an isomorphism of complexes on $\mathcal{A}$
			$$h_*\mathcal{BPS}^{Dol}_{r,n}\cong h'_*\mathcal{BPS}^{Dol}_{r,n'},$$
			inducing
			$$ H_{BPS}^*(\mdol(r,n),\mathbb{Q})\cong  H_{BPS}^*(\mdol(r,n'),\mathbb{Q})$$
			preserving Hodge structures and the perverse filtration. 
		\end{thm}
		
		It is now natural to wonder whether there is a formulation of the P=W conjecture for BPS cohomology and what would be the relation with P=W and PI=WI. This is understood if, in place of moduli spaces, one considers stacks.
		
		\subsubsection{All P=W conjectures at once}
		Let $\mathfrak{M}_{Dol}(r,n)$ and $\mathfrak{M}_B(r,n)$ be the stacks of respectively semistable Higgs bundles on $\Sigma$ and semisimple representations of its fundamental group into $\mathrm{GL}_r(\mathbb{C})$.
		
		In \cite{Davisonnahtstacks}, Davison conjectured the existence of an isomorphism $\Upsilon$ between the Borel-Moore homology groups of $\mathfrak{M}_{Dol}(r,0)$ and $\mathfrak{M}_B(r,0)$
		$$H^*_{\mathrm{BM}}(\mathfrak{M}_{Dol}(r,0))\cong H^*_{\mathrm{BM}}(\mathfrak{M}_{B}(r,0))$$
		and formulated a stacky version of the P=W conjecture, involving the natural weight filtration $W^{\bullet}$ on $H^*_{\mathrm{BM}}(\mathfrak{M}_B(r,0))$ and a suitably defined perverse filtration $\mathfrak{P}^{\bullet}$ on $H^*_{\mathrm{BM}}(\mathfrak{M}_{Dol}(r,0))$, see \cite[\S 1.3]{Davisonnahtstacks}.
		This is usually referred to as the stacky P=W conjecture PS=WS (we omit shifts for ease of the reader and refer to \cite[Conjecture B]{Davisonnahtstacks} for a precise statement).
		\begin{conj}[PS=WS conjecture]\label{psws}
			There exists an isomorphism
			$$\Upsilon: H^*_{\mathrm{BM}}(\mathfrak{M}_B(r,0),\BQ)\xrightarrow{\simeq}H^*_{\mathrm{BM}}(\mathfrak{M}_{Dol}(r,0),\BQ)$$
			such that
			$$\Upsilon(W_{2i}H^*_{\mathrm{BM}}(\mathfrak{M}_{B}(r,0),\BQ))=\mathfrak{P}_iH^*_{\mathrm{BM}}(\mathfrak{M}_{Dol}(r,0),\BQ)).$$
		\end{conj}
		
		To see how the PS=WS is related with P=W and PI=WI, we need to introduce some notation.	    For any nonzero rational number $\mu$ , $\mu=\dfrac{b}{a}$ for $a,b \in \mathbb{Z}$,  $a>0$ and $\mathrm{gcd}(a,b)=1$, we define
	    $$\mdol^{\mu}:=\sqcup_{k\in \mathbb{Z}_{\geq 0}} \mdol(ka,kb); \quad \mb^{\mu}:=\sqcup_{k\in \mathbb{Z}_{\geq 0}} \mb(ka,kb);$$ 
		$$\mathfrak{M}_{Dol}^{\mu}=\sqcup_{k\in \mathbb{Z}_{\geq 0}} \mathfrak{M}_{Dol}(ka,kb); \quad \mathfrak{M}_{B}^{\mu}=\sqcup_{k\in \mathbb{Z}_{\geq 0}} \mathfrak{M}_{B}(ka,kb).$$
		Similarly, we set
		$$ \mathcal{IC}^{Dol}_{\mu}=\bigoplus_{k\in \mathbb{Z}_{\geq 0}}\mathcal{IC}^{Dol}_{ka,kb};	\quad  \mathcal{IC}^{B}_{\mu}= \bigoplus_{k\in \mathbb{Z}_{\geq 0}}\mathcal{IC}^{B}_{ka,kb}; \quad   \mathcal{BPS}^{Dol}_{\mu}=\bigoplus_{k\in \mathbb{Z}_{\geq 0}} \mathcal{BPS}^{Dol}_{ka,kb}.$$
		
		The following result of Kinjo and Koseki allows to decompose the Borel–Moore homology of $\mathfrak{M}_{Dol}^\mu$ into tensor products of the BPS cohomologies.
		\begin{thm}[\cite{KinjoKoseki2021}, Theorem 5.16]
			$$H^*_{\mathrm{BM}}(\mathfrak{M}_{Dol}^{\mu},\BQ))=H^*\left(\mathrm{Sym}(  H^*(\mathrm{B}\BC^*\otimes \mathcal{BPS}^{Dol}_{\mu}))\right)$$
			\end{thm}
			
		 Later, in \cite[Theorem 1.5]{DHSM2022}, Davison, Hennecart and Schlegel-Mejia proved that the cohomologies of $\mathfrak{M}_{Dol}^\mu$ and $\mathfrak{M}_{B}^\mu$ can be respectively described as the symmetric algebra of free Lie algebras on intersection cohomologies of the corresponding moduli spaces.
		 \begin{thm}[\cite{DHSM2022}, Theorem 1.5]\label{thm:stackic}
		 	$$H^*_{\mathrm{BM}}(\mathfrak{M}_{Dol}^{\mu},\BQ))=H^*\left(\mathrm{Sym}(  H^*(\mathrm{B}\BC^*\otimes \mathrm{Free}_{\mathrm{Lie}}(\mathcal{IC}_{Dol}^{\mu})))\right)$$
		 	$$H^*_{\mathrm{BM}}(\mathfrak{M}_{B}^{\mu},\BQ))=H^*\left(\mathrm{Sym} ( H^*(\mathrm{B}\BC^*\otimes \mathrm{Free}_{\mathrm{Lie}}(\mathcal{IC}_{B}^{\mu})))\right)$$
		 \end{thm}
		 Combining theorems, in the Dolbeault setting this yields the following isomorphism.
		 \begin{cor}\label{cor:bpsic}
		 $$\mathcal{BPS}^{Dol}_{\mu}=\mathrm{Free}_{\mathrm{Lie}}(\mathcal{IC}^{Dol}_{\mu})).$$
		 \end{cor}
		 It is then natural to define 
		 $$ \mathcal{BPS}^{B}_{\mu}:=\mathrm{Free}_{\mathrm{Lie}}(\mathcal{IC}^{B}_{\mu})),$$
		 so that one can rewrite the isomorphisms in Theorem \ref{thm:stackic} as 
		 $$H^*_{\mathrm{BM}}(\mathfrak{M}_{Dol}^{\mu},\BQ))=H^*\left(\mathrm{Sym}_{\oplus} H^*(\mathrm{B}\BC^*\otimes \mathcal{BPS}^{Dol}_{\mu})\right);$$
		 $$H^*_{\mathrm{BM}}(\mathfrak{M}_{B}^{\mu},\BQ))=H^*\left(\mathrm{Sym}_{\oplus} H^*(\mathrm{B}\BC^*\otimes \mathcal{BPS}^{B}_{\mu})\right).$$
	     This result is crucial in constructing the isomorphism between the Borel--Moore homology of the stacks: in fact, the non abelian Hodge isomorphism between the corresponding moduli spaces, induces an isomorphism 
	     \begin{equation}\label{eq:isoic}
	     	\mathcal{IC}^{Dol}_{\mu}\cong \mathcal{IC}^{B}_{\mu}.
	     	\end{equation}
	     The above isomorphism yields an isomorphism between the corresponding Free Lie algebras 
	     \begin{equation}\label{eq:isobps}
	     	\mathcal{BPS}^{Dol}_{\mu}\cong \mathcal{BPS}^{B}_{\mu}.
	     	\end{equation}
	    This in turn yields
	     \begin{equation}\label{eq:isostacks}
	     	H^*_{\mathrm{BM}}(\mathfrak{M}_{Dol}^{\mu},\BQ))\cong H^*_{\mathrm{BM}}(\mathfrak{M}_{B}^{\mu},\BQ))
	     	\end{equation}
	     proving the first part of PS=WS.
	    
	    \begin{rmk}[PS=WS $\Leftrightarrow$ PI=WI]
	    	Theorem \ref{thm:stackic} implies that the intersection cohomologies of the moduli spaces $\mdol(r,0)$ and $\mb(r,0)$ appear as direct summands respectively of $H^*_{\mathrm{BM}}(\mathfrak{M}_B(r,0),\BQ)$ and $ H^*_{\mathrm{BM}}(\mathfrak{M}_{Dol}(r,0),\BQ)$.
	    	Moreover, the isomorphism \eqref{eq:isostacks} maps the intersection cohomology of $\mdol(r,0)$ in that of $\mb(r,0)$. In \cite[Theorem 6.1]{Davisonnahtstacks} Davison shows that the inclusions of the intersection cohomology of the moduli spaces in the Borel--Moore cohomology of the corresponding stacks respect both the perverse and the weight filtration. Hence PS=WS implies PI=WI.
	    	Using the generation result of Theorem \ref{thm:stackic}, Davison shows that also the converse is true, thus establishing the equivalence between PS=WS and PI=WI.	       \end{rmk}

		\begin{rmk}[Betti $\chi$-independence $\Rightarrow$ (PS=WS $ \Leftrightarrow $ P=W)]
			Let $\mathrm{BPS}^{Dol}_{\mathrm{Lie},r,n}$ (resp. $\mathrm{BPS}^{B}_{\mathrm{Lie},r,n}$) be the vector space given by the global sections of the component of bidgree $(r,n)$ in $\mathcal{BPS}^{Dol}_\mu$ (resp. $\mathcal{BPS}^{B}_\mu$). 
			Observe that, on the one hand, the non abelian Hodge isomorphism combined with \eqref{eq:isobps}, yields an isomorphism 
			$$\mathrm{BPS}^{Dol}_{\mathrm{Lie},r,n}\cong \mathrm{BPS}^{B}_{\mathrm{Lie},r,n}.$$
			On the other hand,  the $\chi$-independence Theorem \ref{chiindepnoncoprime} implies that there is an isomorphism 
			$$\mathrm{BPS}^{Dol}_{\mathrm{Lie},r,n}\cong \mathrm{BPS}^{Dol}_{\mathrm{Lie},r,n'}$$
			for all $r,n,n'$, such that the perverse filtration is preserved. 
			In \cite[Theorem 14.10]{DHSM2022} the authors show that there is an analogous isomorphism of vector spaces 
			\begin{equation}\label{eq:bpsbettirn}
				\mathrm{BPS}^{B}_{\mathrm{Lie},r,n}\cong \mathrm{BPS}^{B}_{\mathrm{Lie},r,n'},
			\end{equation}
		    on the Betti side. However it is not known whether this isomorphisms preserves the weight filtration. 
			If this were true, it would imply the equivalence of P=W, PI=WI and PS=WS. \\
			Despite conjectural, the fact that weight filtration are preserved by \eqref{eq:bpsbettirn}, has more than one evidence.
		    In the coprime case, it is implied by the proof of the P=W conjecture and the fact the corresponding isomorphism on the Dolbeault side preserves the perverse filtration. 
		    Moreover, it is shown in \cite{HauselRodriguez-Villegas2008}, that the E-polynomials (recording weights, but taking alternating sums over cohomological degrees) of both sides of \eqref{eq:bpsbettirn} coincide.
		    		\end{rmk}
\section{P=C phenomena and the proof of P=W by Maulik, Shen and Yin}\label{sec:pc}

\subsection{P=C phenomena and overview of the proof }	
		The proof of P=W by Maulik, Shen and Yin \cite{MSY} is obtained as a consequence of a more general result on a class of morphisms, which the authors call \textit{dualizable abelian fibrations}, that are modelled on abelian schemes and compactified Jacobian fibrations.
		This approach consists of two main points.
		\begin{itemize}
			\item The definition of a set of features characterizing dualizable abelian fibrations and the phrasing of a condition called \ita{Fourier Vanishing}, which ensures that this class of morphisms enjoys remarkable properties such as multiplicativity of perverse filtration.
			\item The proof that several classes of maps related to compactified Jacobian fibrations are dualizable abelian fibrations.
			\end{itemize}
		
		Loosely speaking, a dualizable abelian fibration is an abelian fibration $\pi:M\rightarrow B$ of relative dimension $g$ with a  dual fibration $\pi^{\vee}:M^{\vee}\rightarrow B^{\vee}$ satisfying two main properties:
		\begin{enumerate}[(i)]
			\item (Fourier-Mukai) $M$ and $M^{\vee}$ are related via a Fourier-Mukai (FM) transform with nice properties mimicking those of dual abelian schemes;
			\item (Full support) All the simple perverse summands in the decomposition theorem of $R\pi_*\mathbb{Q}_M$ are supported on the whole base $B$.
			\end{enumerate}
	    The Chern character of the FM kernel induces a Fourier transform in cohomology $$\mathfrak{F}=\sum \mathfrak{F}_k:H^*(M^{\vee},\mathbb{Q})\rightarrow H^*(M,\mathbb{Q}), \quad \text{such that } \mathfrak{F}_k(H^i(M^{\vee},\mathbb{Q})) \subseteq H^{\geq i+2k-2g}(M,\mathbb{Q}),$$
	    and a filtration $$C^{\mathfrak{F}}_kH^*(M,\mathbb{Q})=\mathrm{Span}\{Im\mathfrak{F}_j\mid j\leq k\}.$$ 
	    called the \textit{Chern filtration} associated with $\mathfrak{F}$.

        \begin{example} 
Any abelian scheme $\pi:M\rightarrow B$ is a dualizable abelian fibration. While Fourier transforms can be described in terms of the multiplication map on $M$, the full support condition is a consequence of the fact that $\pi$ is a smooth proper morphism and thus the decomposition theorem reduces to Deligne's theorem for smooth projective families, see Appendix \ref{sec:appendix}.
\end{example}

       The main result in Maulik-Shen-Yin's paper is the following theorem, see \cite[Theorem 0.1]{MSY}, establishing multiplicativity of the perverse filtration and the inclusion of the Chern filtration in the perverse one for all dualizable abelian fibration satisfying a technical condition on the Fourier-Mukai transforms called \ita{Fourier Vanishing} (FV).
       \begin{thm}\label{thm:pcfordab}
       	Let $\pi:M\rightarrow B$ be a dualizable abelian fibration satisfying $(\mathrm{FV})$. Then 
       	\begin{enumerate}[1.]
       		\item (Multiplicativity of $P$) The perverse filtration $P_\bullet$ on $H^*(M,\mathbb{Q})$ associated with $\pi$ is multiplicative;
       		\item (Perverse $\supseteq$ Chern) For any class $\alpha\in H^*(M^{\vee},\mathbb{Q})$ 
       		$$ \mathfrak{F}_k(\alpha)\in P_kH^*(M,\mathbb{Q}).$$
       		\end{enumerate}
       	\end{thm}
       
       \begin{rmk}
       	Given a dualizable abelian fibration $\pi:M\rightarrow B$ and the associated cohomological operators $\mathfrak{F}_i$, the $(\mathrm{FV})$ condition prescribes that 
       	$$ \mathfrak{F}_i^{-1}\circ \mathfrak{F}_j=0 , \text{ for } i+j<2g.$$
       	This condition, which is in general hard to verify, ensures that the operators $\mathfrak{F}_i$ provides \textit{projectors} on the cohomology of $M$, which recover the perverse filtration on $H^*(M,\mathbb{Q})$. Moreover, the presence of projectors together with $(\mathrm{FV})$ reduces the statement (Perverse $\supseteq$ Chern) to a simple calculation, cfr. proof of Theorem \ref{thm:theorem0.2}.
       	\end{rmk}

	   Maulik, Shen and Yin prove that, given a family $C\rightarrow B$ of integral projective locally planar curves which admits a section and defined $\pi_d:\cjac^d\rightarrow B$ to be the associated compactified Jacobian fibration, if the total space of $\cjac^d$ is nonsingular then $\pi_d$ satisfies conditions (i) and (ii) of dualizable abelian fibrations (see \cite[Theorem 0.2]{MSY}), together with Fourier Vanishing. 
	    Indeed the theory of Fourier-Mukai transforms for compactified Jacobians has been developed by Arinkin \cite{Arinkin2011,Arinkin2013}, while the full support in the decomposition theorem of $R\pi_{d*}\mathbb{Q}_{\cjac^d}$ is ensured by a combination of Ng\^{o} support theorem \cite{ngo} and Severi identitities \cite[Lemma 4.1]{MaulikShen2023}.
       Later, they generalize the result to families without a section.
       
       \begin{cor}[\cite{MSY}, Corollary 4.5]\label{cor:4.5intro} Let $C\rightarrow B$ be a family of projective integral locally planar curves and let $\pi_d: \cjac^d\rightarrow B$ the relative compactified Jacobian. If the total space of $\cjac^d$ is nonsingular, the statements in Theorem \ref{thm:pcfordab} hold for $\pi_d: \cjac^d\rightarrow B$.
       \end{cor}

       Consider now the Hitchin map $h:\mdol\rightarrow \mathcal{A}$ and restrict it to the elliptic locus $B\subset\mathcal{A}$ parametrising integral spectral curves. 
        On $B$, the Hitchin fibration is nothing but the relative compactified Jacobian fibration of the spectral family $C\rightarrow B$:
	    $$\pi_d:\cjac^d\rightarrow  B,$$
        where the degree $d$ of the compactified Jacobian is determined by $r$ and $n$ via Riemann--Roch formula.
    
	   In this setting, if we restrict the tautological generators $c_k(\gamma)$ introduced in section \ref{sec:theoriginalconjecture} to $\cjac^d$, the filtration defined via the Chern character of a universal family on $\mdol(r,n)$
       \begin{equation}\label{eq:chernintro}
       C_kH^*(\cjac^d,\mathbb{Q})=\mathrm{Span}\left\lbrace c_j(\gamma)\mid j\leq k \right\rbrace
       \end{equation}
       is determined by the Chern filtration associated with the Fourier--Mukai transforms, see section \ref{sec:curvesonsurfaces}.
 
  As an application of Corollary \ref{cor:4.5intro}, Maulik, Shen and Yin  provide a new proof of P=W. In fact the P=W conjecture can be decomposed into three identities:	$$ P_k(H^*(\mdol,\mathbb{Q}))=C_k(H^*(\mdol,\mathbb{Q}))=C_k(H^*(\mb,\mathbb{Q}))=W_{2k}(H^*(\mb,\mathbb{Q})).$$ 
  The second identity has been established earlier by works of Markman \cite{Markman2002} and  Hausel–Thaddeus \cite{HauselThaddeus04}, whereas the third by Shende \cite{Shende17}.
  Hence, one is left with proving the first identity  P=C on $\mdol$. The Curious Hard Lefschetz theorem allows to reduce the proof of the equality to just showing that 
$$ C_kH^*(\mdol,\mathbb{Q})\subseteq P_kH^*(\mdol,\mathbb{Q}).$$      	
 Now, Corollary \ref{cor:4.5intro} ensures that C $\subseteq$ P  on the elliptic locus of the Hitchin map. 
The result can now be extended to the whole Hitchin base $\mathcal{A}$ via the same argument as in \cite[Section 8]{HMMS2022}.

\begin{rmk} The above discussion suggests that P=C phenomena might make sense in a broader context. For example, there is no Betti version of $D$-twisted Higgs moduli spaces, so we cannot ask about P = W, but we can ask whether P = C holds.
Moreover, the streamline of the proof suggests that P=C is rather a property of compactified Jacobian fibrations associated with family of curves and does not depend on the representation theory perspective of the other proofs of P=W.
\end{rmk}

\begin{rmk} In the same way as the P=W conjecture, it is possible to establish a P$\supseteq$ C phenomenon analogous to Theorem \ref{thm:pcfordab} for the moduli space $M_{r,\chi}$ of 1-dimensional stable sheaf on $\mathbb{P}^2$ having support on [$rH$] and Euler Characteristic $\chi$ \cite[Theorem 0.6]{MSY}. Here $H$ denotes the hyperplane class.
It is conjectured that also the other inclusion is true, giving a $P=C$ phenomenon. The conjecture has been recently verified in some cases in \cite{KLMP} and it is expected to hold for more general Del Pezzo surfaces, see also \cite{KPS}.
\end{rmk}

In what follows, for expository purposes, we first show how Fourier--Mukai transforms control the perverse filtration in the case of abelian schemes, then we present a sketch of the proof of Theorem \ref{thm:pcfordab} in the case of the relative compactified Jacobian associated with families $C\rightarrow B$ with a section. Later, by promoting schemes to gerbes, one extends Theorem \ref{thm:theorem0.2} to (twisted) compactified Jacobians associated to families without a sections. Finally, we show how to deduce the P=W conjecture from it. 

\subsection{Perverse=Chern for abelian schemes}\label{sec:pcabelian}
Suppose $\pi:A\rightarrow B$ is a smooth abelian scheme of relative dimension $g$.
Since $A$ has a group structure, one can consider multiplication by $[N]:A\rightarrow A$ inducing
$$ [N]^*:H^*(A,\mathbb{Q})\rightarrow H^*(A,\mathbb{Q})$$
in cohomology. The eigenspace decomposition of $[N]^*$ is called \textit{Beauville decomposition} and reads as 
\begin{equation}\label{eq:Beauvilledec}
H^*(A,\mathbb{Q})=\bigoplus_{i=0}^{2g} H^*(A,\mathbb{Q})_{(i)},
\end{equation}
where $H^*(A,\mathbb{Q})_{(i)}=\left\lbrace \alpha \in H^*(A,\mathbb{Q})\mid [N]^*\alpha=N^i\alpha\right\rbrace.$
Let us point out two key properties of the Beauville decomposition.
\begin{itemize}
\item (Splitting of Leray filtration) The Beauville decomposition splits the Leray filtration associated with $\pi$. Note that, since $\pi$ is a smooth projective morphism, the Leray and the Perverse Leray filtration coincide. 
\item (Strong multiplicativity) The Beauville decomposition is multiplicative with respect to cup product, i.e.
\begin{equation}\label{eq:strongmult} H^d(A,\mathbb{Q})_{(i)}\times H^{d'}(A,\mathbb{Q})_{(j)} \rightarrow H^{d+d'}(A,\mathbb{Q})_{(i+j)}.\end{equation}
\end{itemize}
Beauville \cite{Beauville86} shows that there is a way to retrace the decomposition in \eqref{eq:Beauvilledec} via Fourier transforms as follows. 
There is a Poincaré line bundle $\mathcal{P}\rightarrow A\times_B A^{\vee}$, which we think of as an element of $D^bCoh(A\times_B A^{\vee})$. 
Its Chern character $\mathrm{ch}(\mathcal{P})\in H^*(A\times_B A^{\vee},\mathbb{Q})$ defines a Fourier--Mukai transform
\begin{equation}\label{eq:poincaré}
\mathfrak{F}:H^*(A^{\vee},\mathbb{Q})\rightarrow H^*(A,\mathbb{Q}), \quad \alpha \mapsto p_{2*}(p_{1}^*\alpha \cup \mathrm{ch}(\mathcal{P})),
\end{equation}
which admits an inverse $\mathfrak{F}^{-1}:H^*(A,\mathbb{Q})\rightarrow H^*(A^{\vee},\mathbb{Q})$ (i.e. $\mathfrak{F}^{-1}\circ \mathfrak{F}=\mathcal{O}_{\Delta/B}$).

Beauville shows the the decomposition \eqref{eq:Beauvilledec} enjoys the following properties. 

\begin{enumerate}[1)]
\item (Fourier stability) Calling $\mathfrak{F}_i$ the component of $\mathfrak{F}$ given by cupping with $\mathrm{ch}_i(\mathcal{P})$ in \eqref{eq:poincaré}, we have mappings
\begin{center}
\begin{tikzcd}[swap,bend angle=45]
H^d(A,\mathbb{Q})_{(i)}\ar[bend left, "\mathfrak{F}^{-1}_{2g-i}",']{r}& H^{d+2g-2i}(A^{\vee},\mathbb{Q})_{(2g-i)}. \ar[bend left, "\mathfrak{F}_{i}"]{l}
\end{tikzcd}
\end{center}
In particular, when applying the Fourier transform on each summand of the Beauville decomposition, there is just one degree of the Chern character wich contributes. 
Also, note that the Chern grading of $\mathfrak{F}$ affects also the cohomological degree of the image.
\item (Strong multiplicativity) Given the addition map $\mu:A^{\vee}\times_B A^{\vee}\rightarrow A^{\vee}$ on $A^{\vee}$, we can define the \textit{Pontryagin convolution product}
$$*: H^*(A^{\vee},\mathbb{Q})\times H^*(A^{\vee},\mathbb{Q})\rightarrow H^*(A^{\vee},\mathbb{Q}), \quad \alpha* \beta:= \mu_*(p_1^*\alpha\cup p_2^*\alpha)$$ 
whose convolution kernel is supported on the graph of $\mu$ (and thus it is supported in codimension $g$ in $A^{\vee}\times_B A^{\vee}\times_BA^{\vee}$).
Thanks to Fourier stability and the convolution product, we can get a new proof of the strong multiplicativity \eqref{eq:strongmult}. In fact, given $\alpha\in H^d(A,\mathbb{Q})_{(i)}$ and  $\beta\in H^{d'}(A,\mathbb{Q})_{(j)}$, their cup product $\alpha\cup \beta\in H^{d+d'}(A,\mathbb{Q})_{?}$ for a given index \textquotedblleft ?\textquotedblright\ which we want to determine. If we apply the $\mathfrak{F}^{-1}$, we can easily verify that  
$$\mathfrak{F}^{-1}(\alpha\cup \beta)=\mathfrak{F}^{-1}(\alpha)*\mathfrak{F}^{-1}(\beta).$$
On the one hand, by Fourier stability, one knows that the first member must lie in $H^{d+d'+2g-2?}(A^{\vee},\mathbb{Q})_{(2g-?)}$. On the other hand, the properties of the convolution product and the codimension of the support of the convolution kernel imply that the second member of the equality lies in $H^{d+d'+2g-2(i+j)}$, hence that $?=i+j$.

\item (Perverse = Chern) One has that 
\begin{equation}\label{eq:P=Cabelian}
\mathfrak{F}_i(H^*(A^{\vee},\mathbb{Q}))=H^i(A,\mathbb{Q})_{(i)}
\end{equation}
\end{enumerate}

\begin{rmk}
Among the three properties above, the only one which is proved using the existence of multiplication by $[N]$ is  Fourier stability. In particular, one can show that rescaling by different powers of $[N]^*$ yields the following vanishing condition:
\begin{equation}\label{eq:FVabelian}
\mathfrak{F}^{-1}_i\circ \mathfrak{F}_j=0 \ \forall \  i+j\neq 2g.
\end{equation}
Once \eqref{eq:FVabelian} is established, the strong multiplicativity and Perverse = Chern follow as a formal direct consequences. 
\end{rmk}

The upshot that one should get from the above discussion is that the Beauville decomposition, which is a splitting of the (perverse) Leray filtration, is governed by the Chern grading of the Fourier transforms associated with the Poincaré line bundle on $A^{\vee}\times_B A^{\vee}$. 
It is then natural to wonder how to extend the abelian scheme picture to settings where singular fibers may appear, such as the cases of the relative compactified Jacobians, Mukai and Hitchin systems.
The main difficulty in extending Beauville work is that it is not possible to extend the multiplication by $[N]$ to singular fibers. However, thanks to the work of Arinkin \cite{Arinkin2011,Arinkin2013}, it is possible to extend Fourier transforms in the case of dualizable abelian fibrations. 

In what follows, we present an extension of Beauville results to the case of the compactified Jacobian fibration as it is one of the key examples and it is a crucial step in the proof of the P=W conjecture. Indeed Maulik, Shen and Yin prove a theorem which generalizes the above notions for all dualizable abelian fibrations, see \cite[Theorem 2.4]{MSY} (see also \cite{MSY24} for a refinement of this result). Since their strategy involves passing from coherent sheaves to cycle classes, the authors use the language of \ita{relative Chow Motives} developed by Corti and Hanamura \cite{CortiHanamura}. In fact, this language has the advantage of admitting various realizations: Chow groups, mixed Hodge modules and global cohomology endowed with a mixed Hodge structure.
For expository purposes, we present the results in terms of their cohomological realization and we refer to \cite[\S 2]{MSY} for further details.

\subsection{Compactified Jacobians with a section}\label{sec:jacwithsec}

We work in the following setup. Let $C\rightarrow B$ be a family of projective integral locally planar curves such that there is a section $B\rightarrow C$. 
Let $\pi:\cjac\rightarrow C$ be the corresponding relative compactified Jacobian fibration and suppose also to choose $B$ so that the total space $\cjac$ is nonsingular.
This latter request is not a restrictive hypothesis: fix a point $b_0\in B$ such that the corresponding curve $C_{b_0}$ is singular and let $(\mathbb{V},0)$ be the product of the versal deformations of the singularities of $C_{b_0}$. To ensure that the total space of $\cjac$ is smooth around  $C_{b_0}$, it is sufficient that the image $\mathcal{T}$ of $T_{b_0}B$ in $T_0\mathbb{V}$ is transverse to the image of $T_x \overline{J}_{C_{b_0}}$ for all $x\in \overline{J}_{C_{b_0}}$. This is achieved for instance if $\dim \mathcal{T}$ is greater than the cogenus of $C_{b_0}$ (see \cite{FGvS} for more details).

On the regular locus of $B$, the map $\pi$ restrict to an abelian scheme satisfying Beauville decomposition. 
The following result by Arinkin provides a Cohen-Macaulay extension of the Poincaré bundle on the regular locus to the boundary of the family.
\begin{prop}[Arinkin] Let $\pi:\cjac\rightarrow C$ be a relative compactified Jacobian as above. Then there exists an element $\mathcal{P}\in D^bCoh(\cjac\times_B\cjac)$, extending the Poincaré bundle on the regular locus \footnote{More precisely, $\mathcal{P}$ is the direct image by the open embedding $j: J_C\times_B \cjac \cup  \cjac\times_B J_C\hookrightarrow \cjac\times_B \cjac$ of the Poincaré bundle on $J_C\times_B \cjac$.}, which defines a Fourier--Mukai transform 
$$\mathfrak{F}: H^*(\cjac,\mathbb{Q})\rightarrow H^*(\cjac,\mathbb{Q})$$
admitting an inverse $\mathfrak{F}^{-1}$.
\end{prop}

It is not difficult to verify that, unlike in the abelian scheme case, the vanishing 
\begin{equation}
\mathfrak{F}^{-1}_i\circ \mathfrak{F}_j=0 \quad \forall \ i+j\neq 2g
\end{equation}
no longer holds and thus that the image of a cohomology class $\alpha\in H^*(\cjac,\mathbb{Q})$ under $\mathfrak{F}$ splits into different degrees of the perverse filtration of $\pi$.

Thanks to a dimension estimate on the support of $\mathcal{P}$, Maulik, Shen and Yin manage to prove the vanishing for $i+j<2g$. The idea of the proof consist in replacing multiplication by $[N]$ by a more complicated strucure involving again $\delta$-inequality and Adams operations, see \cite[\S 3.5]{MSY} for a detailed account. 
\begin{prop}\label{prop:halfvanishing} Keep the notation as above. The Fourier--Mukai transform $\mathfrak{F}$ satisfies
$$\mathfrak{F}^{-1}_i\circ \mathfrak{F}_j=0 \quad \forall \ i+j<2g.$$
\end{prop}

With half of the vanishings, it is then possible to define operators 
\begin{equation}\label{eq:halfvanishing}
\mathfrak{p}_k:=\sum_{i\leq k} \mathfrak{F}_i\circ \mathfrak{F}^{-1}_{2g-i}, 
\end{equation}
such that
\begin{equation}\label{eq:projectorspk}
\mathfrak{p}_k\circ \mathfrak{p}_k=\mathfrak{p}_k; \quad \mathfrak{p}_{k+1}\circ \mathfrak{p}_k=\mathfrak{p}_k.
\end{equation}
Thanks to the operators $\mathfrak{p}_k$ one can extend the results of section \ref{sec:pcabelian} as follows.
\begin{thm}[\cite{MSY}, Theorem 0.2]\label{thm:theorem0.2}
Let $C\rightarrow B$ be a family of integral locally planar curves with a section and let $\pi:\cjac\rightarrow B$ be the associated compactified Jacobian fibration. Suppose further that $\cjac$ is nonsingular.
\begin{enumerate}
\item (Realization of $P$ via Fourier--Mukai) Setting, for each $k$
$\mathfrak{p}_k:=\sum_{i\leq k} \mathfrak{F}_i\circ \mathfrak{F}^{-1}_{2g-i}$ one has
$$P_kH^*(\cjac,\mathbb{Q})= \mathrm{Im}\mathfrak{p}_k\subset H^*(\cjac,\mathbb{Q});$$
\item (Multiplicativity)
for each $k,l$ one has
$$\cup: \ P_kH^*(\cjac,\mathbb{Q})\times P_kH^*(\cjac,\mathbb{Q})\rightarrow P_{k+l}H^*(\cjac,\mathbb{Q});$$
\item (Perverse $\supseteq$ Chern) for all $k$
$$\mathfrak{F}_k(H^*(\cjac,\mathbb{Q}))\subseteq P_kH^*(\cjac,\mathbb{Q}).$$
\end{enumerate}
\end{thm}
\begin{proof}(Sketch)
	Let 
	$$P'_kH^*(\cjac,\mathbb{Q}):= \mathrm{Im}\mathfrak{p}_k.$$
	The first condition in \eqref{eq:projectorspk} implies that $\mathfrak{p}_k$ is a projection, while the second condition ensures that $P'_{\bullet}$ is an increasing filtration on the cohomology of $\cjac$:
	$$P'_0H^*(\cjac,\mathbb{Q})\subseteq P'_1H^*(\cjac,\mathbb{Q})\subseteq \ldots \subseteq P'_{2g}H^*(\cjac,\mathbb{Q}).$$
	Moreover, as a direct consequence of the vanishing in \eqref{eq:halfvanishing} one has
	\begin{equation}\label{eq:fourierstabp} \mathfrak{F}^{-1}(P'_k H^d(\cjac,\mathbb{Q}))\in H^{\geq d+2g-k}(\cjac,\mathbb{Q})) ,\ \forall \ k \end{equation}
	which can be thought as the generalization of Fourier stability.
\begin{enumerate}[1.]
\item  We want to prove that the filtration $P'$ defined above is indeed the perverse filtration associated with $\pi$. This is done by Ng\^o support theorem: since the properties which imply Proposition \ref{prop:halfvanishing} can be checked étale locally and the perverse cohomology groups $^\mathfrak{p}\mathcal{H}^i$ have full support, to prove that $P$ and $P'$ coincide, one essentially needs to show that $P'$ recovers the Leray filtration on the regular locus of the family, and this holds since the restriction of the $\pi:\cjac\rightarrow B$ to the regular locus is an abelian scheme. The full support condition implies that $P'$ coincides with the perverse Leray filtration $P$ on the whole family. 

\item As in the abelian case, the proof of the multiplicativity involves compatibility of $\mathfrak{F}$ with a convolution product. Again, Arinkin shows that it is possible to extend the convolution product of the smooth locus of the family (which is an abelian scheme) so that the support of the convolution kernel is still of codimension $g$ and 
\begin{equation}\label{eq:convolutionjac}
*:H^d(\cjac,\mathbb{Q})\times H^{d'}(\cjac,\mathbb{Q})\rightarrow H^{\geq d+d'-2g}(\cjac,\mathbb{Q}).
\end{equation}
Now, given $\alpha \in P_kH^d(\cjac,\mathbb{Q})$ and $\beta\in P_lH^d(\cjac,\mathbb{Q})$ we have that
$$\mathfrak{F}^{-1}(\alpha\cup \beta)=\mathfrak{F}^{-1}(\alpha)*\mathfrak{F}^{-1}(\beta).$$
Then the multiplicativity of $P_\bullet$ is implied by \eqref{eq:convolutionjac} and \eqref{eq:fourierstabp} following the same argument as in the abelian scheme case.
\item Setting $\mathfrak{q}_k=\sum_{i\geq k+1}\mathfrak{F}_i\circ\mathfrak{F}^{-1}_{2g-i}$, it is not difficult to see that 
$$H^*(\cjac,\mathbb{Q})=\mathfrak{p}_kH^*(\cjac,\mathbb{Q})\oplus \mathfrak{q}_{k+1}H^*(\cjac,\mathbb{Q}).$$
Hence, given $\alpha\in P_kH^*(\cjac,\mathbb{Q})$ it suffices to show that 
	$$ \mathfrak{q}_{k+1}\circ\mathfrak{F}(\alpha)=0.$$
This follows by the definition of $\mathfrak{q}_{k+1}$ and \eqref{eq:halfvanishing}.
\end{enumerate}
\end{proof}

\subsection{Compactified Jacobian without a section}\label{sec:jacwithoutsec}
A subsequent step leading to P=C phenomena and the proof of the P=W conjecture is to extend Theorem \ref{thm:theorem0.2} to the case of relative compactified Jacobians of families $C\rightarrow B$ having smooth total space but equipped only with a multisection of degree $r$ finite and flat over $B$. 
This is for example the case of the Hitchin fibration or the Le Poitier moduli spaces of sheaves on $\mathbb{P}^2$.
In both cases, one considers relative compactified Jacobians $\pi_d:\cjac^d\rightarrow B$ of torsion free semistable sheaves of degree $d$ on the fibers of $C\rightarrow B$.
When there is no honest section $B\rightarrow C$, extending the result of Theorem $\ref{thm:theorem0.2}$ to $\pi_d:\cjac^d\rightarrow B$ is far from immediate and we refer to \cite[\S 4]{MSY} for more details.  In fact there are two main technical obstacles.
\begin{enumerate}[1.]
\item First, the universal sheaf on $C\times \cjac^d$ might not exist.
\item Second, even if it exists, it might not be unique and different choices of a universal family might affect the Poincaré sheaf and thus the associated Fourier--Mukai transform (which should recover the perverse filtration). The presence of a section allows to trivialize the universal family along it, thus eliminating the ambiguity given by a line bundle pulled back from $\cjac^d$.
\end{enumerate}

%
%

\subsection{Proof of the P=W conjecture}

The final step in the proof of the $P=W$ conjecture is to apply the machinery developed for $\pi_d$ to the Hitchin fibration $h$.
Let us summarize how this is achieved.
\begin{itemize}
	\item As mentioned at the beginning of the section, the proof of the P=W conjecture can be reduced to showing that the Chern filtration $C_k$ defined by the Chern character of the universal family on the moduli space is contained in the perverse filtration associated with $h$. Since two universal families on $\mdol\times \Sigma$ differ by the pullback of a line bundle on $\mdol$, we introduce the notion of \ita{normalized Chern characters} so that the filtration is defined $C_k$ without ambiguity, 
    This is explained in section \ref{sec:normtauclasses}.
    \item Then we take  family of integral curves $C\rightarrow B$ lying on a surface $S$ (for example one might think of $S$ as the total space of a bundle over $\Sigma$ and take a linear system of curves inside it) and considers the associated relative compactified Jacobian fibration $\cjac^d$. In section \ref{sec:curvesonsurfaces}, we show that the perverse filtration associated with the map $\mathrm{id}_S\times \pi_\times: S\times \cjac^d\rightarrow S\times B$ is multiplicative and that there are well defined Fourier--Mukai operators $\mathfrak{F}_k$, whose image is in the right perversity to ensure that the Chern $\subseteq$ Perverse.
    \item Finally, in section \ref{sec:modulispaces} we introduce several moduli spaces of Higgs bundles with decorations and we first prove P=C for the moduli space $\mdol^{\mathrm{parell}}$ of parabolic Higgs bundles having integral associated spectral curve and such that the residue $\theta_p$ has distinct eigenvalues over a fixed point $p$. The moduli space $\mdol^{\mathrm{parell}}$ is equipped with a finite group action whose quotient is a relative compactified Jacobian fibration of curves lying in an appropriate surface (namely the total space of the bundle $\Omega_1^(p)$ on $\Sigma$). Applying the results in the previous step, one deduces P=C for $\mdol^{\mathrm{parell}}$ and from it the P=C for $\mdol$. 
   \end{itemize}

\subsubsection{Normalized tautological classes and Chern filtration}\label{sec:normtauclasses}
Fix  integers $r,n$ such that $r>0$, $(r,n)=1$ and denote $\mdol(r,n)$ simply by $\mdol$.
The coprimality condition ensures that $\mdol$ admits a rank $r$ universal family $(\mathcal{U},\theta)$ on $\mdol\times \Sigma$. 
 Given a class 
 $$\delta:=p_\Sigma^*\delta_\Sigma+p_{\mathcal{M}}^*\delta_{\mathcal{M}} \in H^2(\mdol\times\Sigma,\mathbb{Q})$$
 with $\delta_\Sigma\in H^2(\Sigma,\mathbb{Q})$ and $\delta_{\mathcal{M}} \in H^2(\mdol,\mathbb{Q})$,
 one defines the \ita{twisted Chern Character} $\mathrm{ch}^\delta(\mathcal{U})$ as 
 $$\mathrm{ch}^\delta(\mathcal{U}):=\mathrm{ch}(\mathcal{U})\cup \mathrm{exp}(\delta)\in H^*(\mdol\times\Sigma,\mathbb{Q})$$
and denote by $\mathrm{ch}_k^\delta(\mathcal{U})$ its degree $k$ component in $H^{2k}(\mdol\times\Sigma,\mathbb{Q})$.
The class $\mathrm{ch}^\delta(\mathcal{U})$ is said to be \ita{normalized} if, with respect to the K\"unneth decomposition of $H^2(\mdol\times\Sigma,\mathbb{Q})$, one has
$$\mathrm{ch}_1^{\delta}(\mathcal{U})\in H^1(\Sigma,\mathbb{Q})\otimes H^1(\mdol,\mathbb{Q}).$$

Since a universal family $(\mathcal{U},\theta)$ on $\mdol\times \Sigma$ is determined up to twisting by the pullback of a line bundle on $\mdol$, a direct calculation shows normalized $\mathrm{ch}^\delta(\mathcal{U})$ are do not depend on the choice of $(\mathcal{U},\theta,\delta)$. 
We are now in a position to define the tautogical classes as in \eqref{eq:taugen}: for any $\gamma\in H^*(\Sigma,\mathbb{Q})$ and a normalized universal family $(\mathcal{U},\theta, \delta)$ we set
$$c_k(\gamma):=p_{M*}(p_\Sigma^*\gamma\cup \mathrm{ch}_k^{\delta}(\mathcal{U})).$$
The classes $c_k(\gamma)$ generate the cohomology of $\mdol$.
\begin{defn} We define the \ita{Chern filtration} $C_\bullet$ on $H^*(\mdol,\mathbb{Q})$ as 
$$C_kH^*(\mdol,\mathbb{Q}):=\mathrm{Span}\left\lbrace \prod_{i=1}^s c_{k_i}(\gamma)\mid \sum_{i=1}^s k_i\leq k\right\rbrace\subseteq H^*(\mdol,\mathbb{Q}).$$
The integer $\sum_{i=1}^s k_i$ associated with any class in $C_kH^*(\mdol,\mathbb{Q})$ is called its \ita{Chern grading}.
\end{defn}
\begin{rmk} The definition of Chern filtration $C_\bullet$ implies that it is multiplicative with respect to the cup product on $H^*(\mdol,\mathbb{Q})$.
\end{rmk}
We recall that Shende \cite{Shende17} proved that, on the Betti side, one has
$$C_kH^*(\mb,\mathbb{Q})=W_{2k}H^*(\mb,\mathbb{Q}).$$
As a result proving the P=W conjecture amounts to showing P=C. Reducing further, one just needs to prove that P$\supseteq $C as the other inclusion is granted via Curious Hard Lefschetz.  
So far, for compactified Jacobians and their variants we have been able to relate $P$ with the Chern filtration arising from Fourier transforms. 
Since on a dense open set the Hitchin fibration reduces to a family of  compactified Jacobian of integral curves, it is natural to investigate whether the above defined Chern filtration $C_k$ can be expressed in terms of Fourier Transforms. 

\subsubsection{Curves on surfaces and tautological classes}\label{sec:curvesonsurfaces}

Let $S$ be a nonsingular projective surface, and let $d$ be an integer. We assume that $C \to B$ is a flat family of integral curves lying in $S$. Assume $H\subset S$ is a divisor which does not contain any curve $C_b$ in the family. Then $H\subset S$ yields a multisection 
\[
D := \mathrm{ev}^{-1}(H) \subset C \to B
\]
where $\mathrm{ev}$ denotes the evaluation map $\mathrm{ev}: C \to S$ from the total space of $C$ to the surface.
Again, we consider $\pi_d: \cjac^d \to B$ the associated compactified Jacobian, assuming that both~$C$ and $\cjac^d$ are nonsingular. Moreover we denote by $\mathrm{AJ}$ the (stacky) Abel--Jacobi maps\footnote{the second is obtained pulling back the first by the $\mu_r$-gerbe $\cjacg^d\to \cjac^d$.} \[
\mathrm{AJ}: C \to \overline{J}^1_C, \quad \mathrm{AJ}: \CC  \to \overline{\CJ}^1_C,
\]
where $\mathcal{C}$ is a $\mu_r$-gerbe over $C$.
Next, we consider the closed embeddings
\begin{equation}\label{ev/AJ}
\overline{\mathrm{ev}}: = {\mathrm{ev}}\times_B \mathrm{id}_{\overline{J}^d_C}: C \times_B \overline{J}^d_C \to S \times \overline{J}^d_C, \quad \overline{\mathrm{AJ}}:= {\mathrm{AJ}}\times_B \mathrm{id}_{\overline{\CJ}^d_C}: \CC \times_B\overline{\CJ}^d_C \to \overline{\CJ}^1_C \times_B\overline{\CJ}^d_C.
\end{equation} 
We can define Fourier transforms associated with a Poincaré sheaf $\mathcal{P}_{1,d}$ on $\overline{\CJ}^1_C \times_B\overline{\CJ}^d_C$:
\begin{equation*}
\mathfrak{F}=\sum \mathfrak{F}_i \in H^*(\cjacg^1\times_B \cjacg^d), \quad \mathfrak{F}_i \in H^{2(\dim B+2g+i)}(\cjacg^1\times_B \cjacg^d).
\end{equation*}

Now, if there is a universal sheaf $F_d$ on $C\times_B\cjac^d$, this gives a family $\overline{F}_d$ of $1$-dimensional sheaves on $S$ by pushing forward via $\overline{\mathrm{ev}}$: 
\begin{equation}\label{univ_F}
\overline{F}^d:= \overline{\mathrm{ev}}_* F^d \rightsquigarrow S \times \overline{J}^d_C.
\end{equation}

For proving P=C=W one needs to express $\mathrm{ch}(\overline{F}_d)$ in terms of $\mathfrak{F}$.
\begin{prop}[\cite{MSY}, Proposition 5.1]\label{prop:exprchern}
Consider the morphism $\overline{\pi}_d:=\mathrm{id}_S\times \pi_d: S\times \cjac^d\rightarrow S\times B$.
\begin{enumerate}
    \item[(i)]  One has
\[
\mathrm{ch}(\overline{F}^d) = \left( \overline{\mathrm{ev}}_* \left(  \mathrm{exp}(\lambda\cdot p^*_{C}D) \cap  \overline{\mathrm{AJ}}^* \mathfrak{F}\right)\right) \cup  \left(\frac{{l_0}}{1- \mathrm{exp}(-(l_0))} \right) \cup \mathrm{exp}(p_J^*l_J) \in H^*(S\times \overline{J}^d_C, \BQ),
\]
where $\lambda \in \BQ$ is a constant, $l_J \in H^2(\overline{J}^d_C, \BQ)$ and $l_0$ is such that the Todd class with respect to $\overline{ev}$ is given by $\mathrm{td}(T_{\overline{ev}})=\dfrac{\overline{ev}^*l_0}{1-\exp(-\overline{ev}^*l_0)}$.

\item[(ii)] Let $P_\bullet H^*(S\times \overline{J}_C^d, \BQ)$ be the perverse filtration associated with $\overline{\pi}_d$. This perverse filtration is multiplicative and for any class $\beta \in H^*(C, \BQ)$, one has
\[
\overline{\mathrm{ev}}_*\left( 
p_C^*\beta \cap
\overline{\mathrm{AJ}}^* \mathfrak{F}_k   \right)\in P_kH^{\geq 2k+2}(S\times \overline{J}_C^d, \BQ).
\]
\end{enumerate}
\end{prop}

\subsubsection{The moduli spaces}\label{sec:modulispaces}
Fix two coprime integers $n$ and $r$ with $r>0$ and we consider the moduli space $\mdol:=\mdol(r,n)$ of rank $r$ and degree $n$ Higgs bundles on the curve $\Sigma$. Let us also fix $p\in \Sigma$.
We set
\begin{enumerate}
    \item[$\bullet$] $\mdol$: the moduli space of stable Higgs bundles 
    \[
    (E, \theta), \quad  \theta: E \to E \otimes \Omega^1_\Sigma.
    \]
    \item[$\bullet$] $\mdol^{\mathrm{mero}}$: the moduli space of stable meromorphic Higgs bundles 
    \[
    (E, \theta), \quad  \theta: E \to E\otimes \Omega^1_\Sigma(p).
    \]
    \item[$\bullet$] $\mdol^{\mathrm{par}}$: the moduli space of stable parabolic Higgs bundles
    \[
    (E, \theta, F^\bullet),\quad \theta: E \to E \otimes \Omega^1_\Sigma(p), \quad 
    \]
    where $F^\bullet$ is a complete flag on the fiber $E_p$, and the residue $\theta_p$ preserves the flag.
    \item[$\bullet$] $\mdol^0 \subset \mdol^{\mathrm{par}}$: the moduli of parabolic Higgs bundles such  that the residue $\theta_p$ is nilpotent.
    \item[$\bullet$] $\mdol^{\mathrm{parell}} \subset \mdol^{\mathrm{par}}$: the moduli of parabolic Higgs bundles having integral associated spectral curve and such that the residue $\theta_p$ has $n$ distinct eigenvalues over $p$.
    \item[$\bullet$] $\widetilde{\mathcal{M}}_{Dol} \subset \mdol^0$: the moduli of parabolic Higgs bundles with trivial residue at the point $p$, \emph{i.e.}~$\theta_p=0$.
\end{enumerate}

Each of the above moduli spaces admits a proper Hitchin map defined in the obvious way by calculating the characteristic polynomial of the Higgs field, from which we may define the corresponding perverse filtrations. 
Also, for each of them there is a universal bundle $\mathcal{U}$ which allows to define the normalized tautological classes 
\[
c_k(\gamma), \quad k \in \BN, \quad \gamma \in H^*(\Sigma, \BQ)
\]
as in \eqref{eq:taugen}. 

We summarize the relations between moduli spaces above by the diagram
\begin{equation}\label{Hitchin_moduli}
\mdol \xleftarrow{~~f~~} \widetilde{\mathcal{M}}_{Dol} \xhookrightarrow{~~\widetilde{\iota}~~} \mdol^0 \xhookrightarrow{~~\iota_0~~} {\mdol}^{\mathrm{par}} \xhookleftarrow{~~\iota~~} \mdol^{\mathrm{parell}} \xrightarrow{~~q~~} \overline{J}^d_C. 
\end{equation}
While $\widetilde{\iota}, \iota_0, \iota$ are natural inclusions, the map $f: \widetilde{\mathcal{M}}_{Dol} \to \mdol$ is given by forgetting the flag. The symbol $C$ in $ \overline{J}^d_C $ stands for 
\[
C\to W \subset \bigoplus_{i=1}^r H^0\left(\Sigma, \Omega^1_{\Sigma}(p)^{\otimes i}\right),
\]
the family of integral spectral curves in the total space of the bundle  $\Omega^1_\Sigma(p)$ on $\Sigma$ which intersects the fiber over $p$ in $r$ distinct points. 
From this perspective, one can think of $\overline{J}^d_C$ as the open subvariety of $\mdol^{\mathrm{mero}}$ given by the pre-image of $W$ under the corresponding Hitchin map: more precisely, the pre-image is isomorphic to the compactified Jacobian fibration associated with $C \to W$. Finally,  $\mdol^{\mathrm{parell}}$ is equipped with a natural $\mathfrak{S}_r$-action permuting the complete flag: the condition that the eigenvalues of $\theta_p$ are distinct ensures that this action is free. In this way we have a quotient map $q: \mdol^{\mathrm{parell}} \to \overline{J}^d_C$, which fits into the commutative diagram
\begin{equation}\label{Cart}
\begin{tikzcd}
\mdol^{\mathrm{parell}} \arrow[r, "q"] \arrow[d, "h"]
& \overline{J}^d_C \arrow[d, " "] \\
\widetilde{W} \arrow[r]
& W.
\end{tikzcd}
\end{equation}
Here, $\widetilde{W}$ can be interpreted as the parameter space of the spectral curves lying in $W$ which have a marking over $p\in \Sigma$ and whose projection to $W$ is the natural $\mathfrak{S}_r$-quotient. Since the map $h$ in \eqref{Cart} is the pullback of the compactified Jacobian fibration $\overline{J}^d_C$ along a finite \'etale map $\widetilde{W} \to W$, we deduce that the perverse filtration $P_\bullet H^*(\mdol^{\mathrm{parell}},\mathbb{Q})$ is multiplicative with respect to cup product.
As a result, establishing P$\supseteq$ C amounts to show that cupping with a tautological class $c_k(\gamma)$ increases perversity exactly by $k$.
\begin{prop}\label{prop:PCparallel} Let $c_k(\gamma)$ be the tautological classes on $\mdol^{\mathrm{parell}}$. Then
\[
\cup \ c_k(\gamma) : P_iH^*(\mdol^{\mathrm{parell}}, \BQ) \to P_{i+k}H^*(\mdol^{\mathrm{parell}}, \BQ).
\]
\end{prop}
\
Since $P_\bullet H^*( \mdol^{\mathrm{parell}}, \BQ)$ is multiplicative by Corollary \ref{PCgeneral}, one just needs to show that
\[
c_k(\gamma) \in P_k H^*(\mdol^{\mathrm{parell}}, \BQ).
\]
Moreover, since this class is pulled back from $\overline{J}^d_C$, in view of the diagram (\ref{Cart}) it suffices to prove the corresponding statement for $\overline{J}^d_C$.

\par\noindent
\gr{Claim}: $c_k(\gamma) \in P_k H^*(\overline{J}^d_C, \BQ).$

Let $S$ be the total space of projectivization of the bundle $\Omega^1_\Sigma(p)$ over $\Sigma$:
\[
\mathrm{pr}: S:= \BP_\Sigma\left( \Omega^1_\Sigma(p) \oplus \CO_\Sigma\right) \to \Sigma.
\]
Then $C \to W$ can be viewed as a family of curves in the linear system $|r\Sigma|$ with $\Sigma \subset S$ the 0-section. A universal sheaf $\overline{F}^d$ on $S\times \overline{J}^d_C$ of Section \ref{sec:curvesonsurfaces} provides a universal bundle
\[
\mathcal{U}: = (\mathrm{pr} \times \mathrm{id}_J)_* \overline{F}^d \rightsquigarrow \Sigma \times \overline{J}^d_C.
\]
In particular, with the Grothendieck--Riemann--Roch formula with respect to 
\[
\overline{\mathrm{pr}}:=\mathrm{pr} \times \mathrm{id}_J: S\times \overline{J}^d_C \to \Sigma \times \overline{J}^d_C
\]
one can express the class $c_k(\gamma)$ in terms of the tautological classes associated with $S$ defined via $\overline{F}^d$. More precisely, the formula
\[
\mathrm{ch}(\mathcal{U}) = \overline{\mathrm{pr}}_*\left( \mathrm{ch}(\overline{F}^d) \cup p_S^* \mathrm{td}_{\mathrm{pr}}\right),
\]
implies that every tautological class $c_k(\gamma)$ can be written in terms of K\"unneth components of 
\[
\mathrm{ch}^{\overline{\mathrm{pr}}^*\delta}_{j}(\overline{F}^d), \quad j \leq {k+1}.
\]
We are thus reduced to showing  
\[
\mathrm{ch}^{\delta'}_{k + 1}(\overline{F}^d) \in P_{k}H^*(S\times \overline{J}^d_C, \BQ).
\]
with $\delta'=\overline{\mathrm{pr}}^*\delta$.
Now by item (i) of Proposition \ref{prop:exprchern}, one can express $\mathrm{ch}^{\delta'}_{k + 1}(\overline{F}^d)$ in terms of $\mathfrak{F}$ and $l_0$. In particular it will be a sum of terms of the form
$$\left( \overline{\mathrm{ev}}_* \left( \mathrm{exp}(\lambda\cdot p^*_{C}\beta_j) \cap  \overline{\mathrm{AJ}}^* \mathfrak{F}\right)\right) \cup \left( \gamma_j(l_0)\right), \quad \text{ for }j\leq k.
$$
Here $\gamma_j(l_0)$ is a polynomial in $l_0$ of degree $\leq j$.
Note that, by item (ii) of Proposition \ref{prop:exprchern}, we have
$$\left( \mathrm{exp}(\lambda\cdot p^*_{C}\beta_j) \cap  \overline{\mathrm{AJ}}^* \mathfrak{F}\right)\in P_jH^{\geq 2j+2}(S\times \overline{J}_C^d, \BQ).$$
Now, since the component of $l_0$ in $\cjac^d$ is pulled back from $W$, it belongs to $P_0H^2(S\times \cjac^d\mathbb{Q})\subseteq P_1H^2(S\times \cjac^d\mathbb{Q})$, thus 
$$\gamma_j(l_0)\in \bigoplus_{i\leq j} P_iH^{2i}(S\times \cjac^d\mathbb{Q}).$$
By the multiplicativity of $P_\bullet$, one can finally conclude 
$$ \mathrm{ch}^{\delta'}_{k + 1}(\overline{F}^d) \in P_{k}H^*(S\times \overline{J}^d_C, \BQ).$$

This completes the proof of Proposition \ref{prop:PCparallel} and establishes Perverse $\supseteq$ Chern for $\mdol^{\mathrm{parell}}$.

The last step in the proof of Maulik, Shen and Yin consists in reducing the proof of P$\supseteq$ C for $\mdol$ to that for $\mdol^{\mathrm{parell}}$. This is pursued following the strategy developed in the proof of Hausel--Mellit--Minets--Schiffmann \cite[Section 8]{HMMS2022}. We briefly recall it here; see also \cite[\S 5.4.3]{MSY} and \cite[\S 4.3]{surveyciky}. 
As already mentioned in section \ref{sec:theoriginalconjecture}, the P=W conjecture for $\mdol$ can be deduced from the following analogue of Proposition \ref{prop:PCparallel} for the moduli space $\mdol$:
\begin{equation}\label{P=W_final}
c_k(\gamma) \cup {} : P_iH^*(\mdol, \BQ) \to P_{i+k}H^*(\mdol, \BQ).
\end{equation}

If such a statement holds for one of the Hitchin type moduli spaces defined at the beginning of the section, we say that this moduli space satisfies \emph{stronger} P $\supseteq$ C. For the moduli space $\mdol$, the stronger P $\supseteq$ C condition is equivalent to the weaker version due to Markman's generation result \cite{Markman2002}. However for other spaces, these two conditions are not equivalent.

The argument goes by fixing a normalized pair $(\mathcal{U},\delta)$ on $\Sigma\times \mdol^{\mathrm{mero}}$ so that one can define tautological classes whose pullback to the other moduli spaces yield tautological classes on them. Then, knowing that P $\supseteq$ C hold on $\mdol^{\mathrm{parell}}$, one investigates how the stronger P $\supseteq$ C condition changes with respect to the morphisms in \eqref{Hitchin_moduli} from the right end to the left end. Since the P $\supseteq$ C is preserved by all morphisms in \eqref{Hitchin_moduli}, the proof is concluded.

		\newpage
		\appendix
		\section{Intersection cohomology and perverse sheaves}\label{sec:appendix}
		In this appendix we recall the main notions of the theory of perverse sheaves and intersection cohomology, which appear in the paper. For proofs, examples and several enlighting explanations, we refer to \cite{deCataldoMigliorini2009}.
		All cohomology groups are considered with rational coefficient, so we omit it in the notation.\\
		
		Let $Y$ be an algebraic variety. We denote by $D^b_c(Y)$ the bounded derived category of $\mathbb{Q}$-constructible complexes on $Y$. Let $\mathbb{D}\colon D^b_c(Y)\rightarrow D^b_c(Y)$ be the Verdier duality functor.  The two full subcategories 
		\begin{align*}
			{^{\mathfrak{p}}}D^b_{\leq 0}(Y) &:= \left \lbrace K^*\in D^b_c(Y) \mid \dim \mathrm{Supp}(\mathcal{H}^j(K^*))\leq -j \right \rbrace\\
			{^{\mathfrak{p}}}D^b_{\geq 0}(Y) &:=\left \lbrace K^* \in D^b_c(Y) \mid \dim \mathrm{Supp}(\mathcal{H}^j(\mathbb{D}K^*))\leq -j \right \rbrace
		\end{align*}
		define a $t$-structure on $D^b_c(X)$, called \ita{perverse} $t$-structure. The heart \[ \mathcal{P}(Y) := {^{\mathfrak{p}}}D^b_{\leq 0}(Y)\cap {^{\mathfrak{p}}}D^b_{\geq 0}(Y)\] of the $t$-structure is the abelian category of \ita{perverse sheaves}.  
		
		The truncation functors are denoted ${^{\mathfrak{p}}}\tau_{\leq k}\colon D^b_c(Y)\to {^{\mathfrak{p}}}D^b_{\leq k}(Y)$, ${^{\mathfrak{p}}}\tau_{\geq k}\colon D^b_c(X)\to {^{\mathfrak{p}}}D^b_{\geq k}(Y)$, and the perverse cohomology functors are 
		\[{^{\mathfrak{p}}} \mathcal{H}^{k} :={^{\mathfrak{p}}}\tau_{\leq k} {^{\mathfrak{p}}}\tau_{\geq k}\colon D^b_c(Y) \to \mathcal{P}(Y).\]
		
		
		The category $\mathcal{P}(Y)$ is Artinian, thus every $K\in \mathcal{P}(Y)$ admits an increasing finite filtration with quotients simple objects. 
		Simple perverse sheaves are all of the form 
		$$\mathcal{IC}_{\overline{Z}}(\mathcal{L})$$
		where $\mathcal{IC}_{\overline{Z}}(\mathcal{L})$ denotes the intersection complex associated with some locally closed smooth subvariety $Z$ and a simple local system  $\mathcal{L}$  on it. 
		
		\begin{defn}
			The intersection complex $\mathcal{IC}_{\overline{Z}}(\mathcal{L})$ associated with a local system $\mathcal{L}$ is a complex of sheaves on $Z$ which extends $\mathcal{L}[\dim Z]$ and is determined up to unique isomorphism in the derived category by the conditions:
			\begin{itemize}
				\item $\mathcal{H}^j(\mathcal{IC})_{\overline{Z}}(\mathcal{L}))=0 \quad 	\text{ for all } j< -\dim Z$,
				\item $\mathcal{H}^{-\dim Y}(\mathcal{IC}_{\overline{Z}}(\mathcal{L})_{\mid Z})\cong \mathcal{L}$,
				\item $\dim \mathrm{Supp}\mathcal{H}^j(\mathcal{IC}_{\overline{Z}}(\mathcal{L}))<-j, \text{ for all }j\in \mathbb{Z}$,
				\item $\dim \mathrm{Supp}\mathcal{H}^j(\mathbb{D}\mathcal{IC}_{\overline{Z}}(\mathcal{L}))<-j, \text{ for all }j\in \mathbb{Z}$.
			\end{itemize}
		\end{defn}
		
		When $\mathcal{L}=\mathbb{Q}_{U}$ for an open nonsingular Zariski dense subset $U$ of $Y$ we simply denote $\mathcal{IC}_Y(\mathcal{L})$ by $\mathcal{IC}_Y$. Note that, if $Y$ is smooth, then $\mathcal{IC}_Y\cong \mathbb{Q}[\dim Y]$.
		
		\begin{defn}[Intersection cohomology]
			Let $Y$ be an algebraic variety. We set
			$$IH^i(Y)\doteq \mathbb{H}^{i-\dim Y}(\mathcal{IC}_Y)$$
			and we call it the $i$-th intersection cohomology group of $Y$. 
		\end{defn}
		In general, given any local system $\mathcal{L}$ supported on a locally closed subset $Z$ of $Y$, the cohomology groups of $Y$ with coefficients in $\mathcal{L}$ are shifted hypercohomology groups of the intersection complex associated to $\mathcal{L}$:
		$$IH^*(\overline{Y},\mathcal{L})=H^{*-\dim Y}(\overline{Y},\mathcal{IC}_{\overline{Y}}(\mathcal{L})).$$
		
		\begin{prop}[Properties of Intersection cohomology] Let $Y$ be an algebraic variety of dimension $n$. Then its intersection cohomology group enjoys the following properties:
			\begin{enumerate}[(i)]
				\item $IH^i(Y)$ is finite dimensional vector spaces and $IH^i(Y)=0$ for $i\not\in\{0\ldots 2n\}$.
				\item There is a natural morphism $H^d(Y)\rightarrow IH^d(Y)$ which is an isomorphism when $Y$ as at worst finite quotient singularities. This morphism endows $IH^*(Y)$ with a module structure over $H^*(Y)$, but in general $IH^*(Y)$ has no ring structure or cup product.
				\item (\textbf{Poincaré duality}) There is a canonical isomorphism $IH^i(Y)\cong IH^{2n-i}(Y)^{\vee}$ for all $i\in \mathbb{N}$. 
				\item $IH^i(Y)$ carry a natural mixed Hodge structure. If $Y$ is projective, the mixed Hodge structure is pure of weight $i$ and $IH^i(Y)$ admits a Hodge decomposition
				$$IH^i(Y)\otimes\BC\cong \bigoplus IH^{p,q}(Y),\quad \overline{IH^{p,q}(Y)}\cong H^{q,p}(Y).$$
			\end{enumerate}
		\end{prop}
		
		\subsection{The decomposition theorem package}
		
		The crowning result of the theory of perverse sheaves is the Decomposition theorem by Beilinson, Bernstein and Deligne \cite{bbd}. For the rest of the section we state all result not assuming varieties to be necessarily nonsingular; when a variety $X$ is nonsingular then $IH^*(X)$ can by replaced simply by $H^*(X)$.
		
		\begin{thm}[Decomposition theorem] Let $h:X\rightarrow Y$ be a proper algebraic map of complex algebraic varieties. There is an isomorphism in $D^b_c(Y)$
			\begin{equation}\label{dt1}
				Rh_*\mathcal{IC}_X\cong \bigoplus_{i\in \mathbb{Z}} \ ^{\mathfrak{p}}\mathcal{H}^i(Rh_*\mathcal{IC}_X)[-i].
			\end{equation}
			Furthermore, the perverse sheaves $^{\mathfrak{p}}\mathcal{H}^i(Rh_*\mathcal{IC}_X)$ are semisimple, i.e. there exists a stratification of $Y\sqcup Y_{\alpha}$ such that $^{\mathfrak{p}}\mathcal{H}^i(Rh_*\mathcal{IC}_X)$ decomposes as direct sum 
			\begin{equation}
				^{\mathfrak{p}}\mathcal{H}^i(Rh_*\mathcal{IC}_X)\cong \bigoplus_\alpha \mathcal{IC}_{\overline{Y_{\alpha}}}(\mathcal{L}_{\alpha}),
			\end{equation}
			where $\mathcal{L}_{\alpha}$ are semisimple local systems on $Y_\alpha$.
		\end{thm}
		
		The direct sum \eqref{dt1} is finite and $i$ ranges from $-r(h)$ to $r(h)$, where $r(h)$ is defined as 
		$$r(h)=\dim X\times_Y X-\dim X.$$
		The decomposition theorem is understood in combination with the \ita{Relative Hard Lefschetz theorem}. As the name suggests, the Relative Hard Lefschetz theorem stated below is a relative version of Hard Lefschetz theorem and it is closely intertwined with the decomposition theorem as it expresses a symmetry between the summands in \eqref{dt1}.
		\begin{thm}[Relative Hard Lefschetz]
			Let $h:X\rightarrow Y$ be a proper map of algebraic varieties with $X$ quasi-projective and let $\alpha$ be the first Chern class of a hyperplane line bundle on $X$. Then we have isomorphisms
			$$\alpha^i \cup: \	^{\mathfrak{p}}\mathcal{H}^{-i}(Rh_*\mathcal{IC}_X)\xrightarrow{\simeq} \ ^{\mathfrak{p}}\mathcal{H}^{i}(Rh_*\mathcal{IC}_X).$$
		\end{thm}
		
		\subsection{Perverse Leray filtration}
		Let $h:X\rightarrow Y$ be a projective map of algebraic varieties.
		\begin{defn}
			The \ita{perverse Leray filtration associated to $h$} is defined as
		\[P_k IH^*(X)=\mathrm{Im}\left\lbrace H^*(Y,\ ^\mathfrak{p}\tau_{\leq -k}\mathrm{R}h_*\mathcal{IC}_X) \rightarrow H^*(Y,\mathrm{R}h_*\mathcal{IC}_X)\right\rbrace.\]
		\end{defn}
		
		\begin{rmk} In some papers, like in the original one on the P=W conjecture \cite{deCataldoHauselMigliorini2012}, the perverse filtration is defined with a shift so that it ranges from 0 to the cohomological degree: 
				\[P'_k IH^*(X)=P_k H^{*-(\dim X -r(h))}(Y,\mathrm{R}h_*\mathcal{IC}_X[\dim X -r(h)]).\]
		In the case of the Hitchin fibration this amounts to shift by the dimension of the Hitchin base.
			\end{rmk}
		
		\begin{defn} A class $\eta\in H^*(X)$ has perversity $k$ if 
			$$\eta\in P_kH^*(X)\text{ and }\alpha\not\in P_{k-1}H^*(X).$$
		\end{defn}
		\begin{defn}
			A \ita{splitting} of the perverse filtration is a vector space decomposition $IH^*(X)=\bigoplus_{\ell} G_{\ell}IH^*(X)$ such that 
			$$P_kIH^*(X)=\bigoplus_{\ell\leq k}G_{\ell}IH^*(X).$$ 
		\end{defn}

	For ease of the reader, we restate the relative the relative Hard Lefschetz theorem in terms of the perverse Leray filtration. 
		\begin{thm}[Relative hard Lefschetz for the perverse filtration]
			Let $h\colon X\rightarrow Y$ be a proper map of algebraic varieties and let $\alpha\in H^2(X)$ be the first Chern class of a relatively ample line bundle. Then there exists an isomorphism 
			$$	\cup \ \alpha^k \colon  Gr^P_{-k}IH^d(X)\rightarrow Gr^P_{k}IH^{d+2l}(X).$$
		\end{thm}

\section*{Declarations}

		\textbf{Funding:} The author is supported by INdAM GNSAGA.\\
		\par
		\noindent
		\gr{Conflict of interests:} The author declares no conflict of interests.
		
		
		%
		%
		%
		%
		\newpage
		\bibliographystyle{plain}
		\bibliography{main}
		\vspace{2cm}
		\noindent
		Department of Mathematics, Physics and Informatics\\
		\noindent
		University of Modena and Reggio Emilia\\
		\noindent
		{\tt camilla.felisetti@unimore.it}
		
		\end{document}